\documentclass[12pt,oneside,final]{amsart}
\usepackage{stmaryrd,amsthm}
\SetSymbolFont{stmry}{bold}{U}{stmry}{m}{n}
\title{An inverse problem for inhomogeneous Signorini obstacle}
\author{Ziyao Zhao}
\address{Ziyao Zhao: Department of Mathematics and Statistics, University of Helsinki, FI-00014 Helsinki, Finland} 
\email{ziyao.zhao@helsinki.fi}
\IfFileExists{tweakslo.sty}{\usepackage{tweakslo}}{
\usepackage{amssymb,thmtools,mathtools,todonotes,amsmath}
\declaretheorem{theorem,definition,lemma,proposition,corollary,remark}}
\providecommand{\HOX}[1]{\todo[noline]{#1}}

\def\p{\partial}
\def\o2{\overline{O_2}}

\def\R{\mathbb R}

\def\cV{\mathcal V}

\def\cS{\mathcal S}
\def\cR{\mathcal R}
\def\cH{\mathcal H}

\def\oO{\overline{O}}

\def\tH{\widetilde{H}}
\def\tbH{\widetilde{\boldsymbol{H}}}
\def\tr{\text{tr}}
\def\div{\text{div }}

\def\bs{\boldsymbol{\sigma}}
\def\be{\boldsymbol{\varepsilon}}
\def\bn{\boldsymbol{\nu}}
\def\bdeta{\boldsymbol{\eta}}
\def\n{\boldsymbol{n}}
\def\bpsi{\boldsymbol{\psi}}
\def\bphi{\boldsymbol{\phi}}
\def\brho{\boldsymbol{\rho}}

\def\bu{\mathbf{u}}
\def\bv{\boldsymbol{v}}
\def\bw{\boldsymbol{w}}
\def\bV{\boldsymbol{V}}
\def\bf{\boldsymbol{f}}
\def\bg{\boldsymbol{g}}
\def\bh{\boldsymbol{h}}
\def\bH{\boldsymbol{H}}
\def\fR{\mathfrak{R}}
\def\0{\mathbf{0}}
\def\bLambda{\boldsymbol{\Lambda}}
\def\bxi{\boldsymbol{\xi}}
\def\bzeta{\boldsymbol{\zeta}}
\def\iE{\mathit{E}}

\DeclarePairedDelimiter{\norm}{\lVert}{\rVert}
\DeclarePairedDelimiter{\abs}{\lvert}{\rvert}
\DeclarePairedDelimiter{\inner}{\langle}{\rangle}

\DeclarePairedDelimiter{\trinorm}{\interleave}{\interleave}
\newcommand\subsetsim{\mathrel{%
  \ooalign{\raise0.2ex\hbox{$\subset$}\cr\hidewidth\raise-0.8ex\hbox{\scalebox{0.9}{$\sim$}}\hidewidth\cr}}}
  
\begin{document}\maketitle
\begin{abstract}
  This paper investigates the inverse problem of determining a general Signorini obstacle using boundary measurements. We demonstrate that both the shape of the obstacle and the obstacle function can be uniquely determined from solution measurements taken on an arbitrary open subset of the boundary. This result applies to both the scalar and elasticity versions of the Signorini problem.
\end{abstract}
\section{Introduction}
  The Signorini problem, originally formulated by A. Signorini in \cite{Signorini33}, stands as a cornerstone in the study of unilateral contact phenomena in solid mechanics. In its classical formulation, the problem seeks the equilibrium configuration of an elastic body resting on a flat, rigid foundation, where physical impenetrability enforces a strict unilateral constraint on the displacement. However, modern applications necessitate a more general framework that replaces the flat foundation with a surface given by the graph of a function.
  The presence of a non-zero thin obstacle significantly enriches the analytical landscape; because the dynamic transition between contact and separation cannot be globally linearized, the inclusion of an inhomogeneous barrier forces the minimizing solution, alongside its associated free boundary, to adapt to the local geometry, curvature, and regularity of the obstacle \cite{Caffarelli79,GP09,GS14,KRS19}. In this paper, we study the inverse problem of determining the shape of the obstacle and the obstacle function simultaneously from boundary measurements.
  \subsection{Problem setting and Main results}
  We begin by formulating our model problem in the scalar setting, before turning to its counterpart in linear elasticity.
  Let $\Omega \subset \R^n$ be a bounded connected open set with smooth boundary $\p\Omega$. Let the shape of the obstacle be modeled by an open set $O \subset\subset \Omega$ with a smooth boundary $\partial O$ and $\Omega\setminus\oO$ is connected. Let the thin obstacle on $\partial O$ be represented by the function $\varphi \in C^\infty(\partial O)$.
  The scalar Signorini problem for the Laplacian on $\Omega\setminus \oO$ is formulated as finding minimizer $u$ to the functional
  \begin{equation}
    A(u): = \int_{\Omega\setminus \oO} \abs{\nabla u}^2\ dx
  \end{equation}
  on the closed convex set
  \begin{equation}
    \mathcal{K}:= \{u\in H^1(\Omega\setminus\oO)\mid u|_{\p O}\geq \varphi,\ u|_{\p\Omega} =f\}.
  \end{equation}

  The foundational theory regarding the existence and uniqueness of solutions to this variational problem was established by Lions and Stampacchia in \cite{LS67}. Early regularity result by Frehse \cite[Lemma 2.2]{FJ77} demonstrated that the solution to the scalar Signorini problem lies in $H^2$. Soon after Caffarelli \cite{Caffarelli79} showed the $C^{1,\alpha}$ regularity for some $0<\alpha \leq 1/2$. The optimal $C^{1,\frac{1}{2}}$ regularity was subsequently proven by Athanasopoulos and Caffarelli \cite[Theorem 4]{AC04} under the assumption of a zero thin obstacle ($\varphi = 0$), a result that was later generalized to inhomogeneous thin obstacles ($\varphi \neq 0$) by Garofalo and Petrosyan \cite[Theorem 2.3.5]{GP09}. With these regularity results, the minimizer $u$ satisfies the Euler-Lagrangian equations
\begin{equation}
\label{eq:direct_problem_lap}
\begin{cases}
\hfil \Delta u=0 \;\textrm{ in }  \Omega\setminus \oO, \\
\hfil u|_{\partial \Omega}=f, \text{ on }\p\Omega\\
\hfil u\geq \varphi, \;\; \partial_{\bn}u\geq 0, \;\; (u-\varphi)\partial_{\bn} u =0 \textrm{ on } \partial O.    
\end{cases}
\end{equation}
where the normal derivative $\partial_{\bn} u|_{\partial O}$ is taken with respect to the outward unit normal $\bn$ for $\Omega\setminus \oO$ at $\partial O$.
\HOX{Actually the disjoint set here does not make any difference due to the unique continuation. Should I keep this setting or just let $\cR = \cS$?}

Let us denote the outward unit normal for $\Omega\setminus \oO$ at $\p \Omega$ by $\n$ and let $\cS,\cR$ be non-empty open subsets of $\p\Omega$.
we define the local Dirichlet-to-Neumann map
\begin{align}
  \label{eq:def_DN}
  \Lambda : C^\infty_0&(\cS) \to C^0(\cR),\\
  \Lambda f &= \p_{\n} u^f |_{\cR},
\end{align} 
where $u^f$ solves \eqref{eq:direct_problem_lap} with the Dirichlet boundary condition $f\in C^\infty_0(\cS)$.

It was shown in \cite{HLLOZ25} that the shape of a flat obstacle $O$, i.e. $\varphi = 0$, can be uniquely determined using a single boundary measurement up to a natural gauge. However, this is impossible without a prior knowledge that $\varphi = 0$ due to the inadequacy of the data, even if the measurement is made on the whole boundary, see Example \ref{example_lap_1} in Section \ref{sec:examples}.

In this paper, we show that one can determine the inhomogeneous obstacle function $\varphi$ and the shape of obstacle $O$ simultaneously with the Dirichlet-to-Neumann map $\Lambda$.
\begin{theorem}
\label{main-Laplace}
\it Let $\Omega\subset \mathbb{R}^n$ be a bounded connected open set with smooth boundary and $\Gamma$ be an open subset of $\p\Omega$. Let $O_1,O_2\subset\subset \Omega$ be (possibly empty) open subsets with smooth boundary, and $\varphi_1\in C^\infty(\p O_1)$, $\varphi_2\in C^\infty(\p O_2)$.
Assume that $\Omega\setminus \overline{O_1},\Omega\setminus \overline{O_2}$ are connected. 
Let $\Lambda_1$, $\Lambda_2$ be the Dirichlet-to-Neumann map defined in \eqref{eq:def_DN} with obstacles $O_1,\varphi_1$ and $O_2,\varphi_2$, respectively.
If $\Lambda_1=\Lambda_2$, then $O_1=O_2$ and $\varphi_1 = \varphi_2$ .
\end{theorem}


Next we formulate our model problem in linear elasticity system. Let us denote by $\bu:\Omega\setminus \overline{O}\to \R^n$ the displacement vector of the elastic body and by $\be(\bu)=\frac{1}{2}(\nabla \bu+(\nabla \bu)^T)$ the linearized strain tensor.
Then the stress tensor for the Lam\'e system is given by 
\begin{equation}
    \bs(\bu)=2\mu \be(\bu)+\lambda\, \tr(\be(\bu))\boldsymbol{I}_n,
\end{equation}
The classical formulation of the Signorini problem for the isotropic linearized elastic system, with prescribed smooth displacement $\boldsymbol{f}$ on the exterior boundary $\partial \Omega$, is finding minimizer $\bu$ to the functional
\begin{equation}
  \boldsymbol{A}(\bu) : = \int_{\Omega\setminus \oO} \bs(\bu) : \be(\bu) \ dx
\end{equation}
over the closed convex set
\begin{equation}
  \boldsymbol{\mathcal{K}} := \{\bu\in (H^1(\Omega\setminus\oO))^n \mid \bu_\nu|_{\p O}\leq \varphi,\ \bu|_{\p\Omega} =\bf\},
\end{equation}
where $\bu_\nu|_{\p O} :=\bu|_{\p O}\cdot\bn$ is the normal displacement on $\p O$.

The foundational questions of existence and uniqueness for this problem were first resolved by Fichera \cite{Fichera63}, with further comprehensive investigations by Duvaut and Lions \cite{DL76}. The pioneering result in regularity was obtained by Kinderlehrer \cite{KD81}, who demonstrated local $C^{1,\alpha}$ regularity for the two-dimensional setting ($n=2$). This was later generalized to arbitrary dimensions by Schumann \cite{schumann1989}, who achieved $C^{1,\alpha}$ regularity by reducing the coupled Lam\'e system to a scalar equation. The critical threshold of optimal $C^{1,\frac{1}{2}}$ regularity was subsequently established by Andersson \cite{Andersson16} for the homogeneous Lamé system with a flat obstacle. Andersson developed a novel approach that does not rely on the comparison principle and thus directly applicable to systems. More recently, R\"uland and Shi \cite{RS22} extended this optimal $C^{1,\frac{1}{2}}$ regularity to encompass the inhomogeneous Lamé system by leveraging the structural connection between the elastostatic Signorini problem and the obstacle problem for the half-Laplacian. By virtue of these regularity results, the minimizer $\bu$ satisfies the Euler-Lagrangian equations

\begin{equation}
    \label{eq:direct_problem_ela}
    \begin{cases}
    \hfil  \div \bs(\bu)=0, \text{ in }\Omega\setminus \overline{O},\\
    \hfil \bu=\boldsymbol{f}, \text{ on }\p\Omega,\\
    \hfil \bs(\bu)_\tau=0,\ \bu_\nu \leq \varphi,\ \bs(\bu)_\nu \leq 0,\ (\bu_\nu-\varphi)\bs(\bu)_\nu=0,\ \text{ on }\p O,
    \end{cases}
  \end{equation}  
where $\bs(\bu)_\nu=\bs(\bu)\bn\cdot \bn$ is the normal stress, and $\bs(\bu)_\tau=\bs(\bu)\bn-\bs(\bu)_\nu \bn$ is the tangential stress (being zero due to no friction).
We define the corresponding local Dirichlet-to-Neumann operator as
\begin{align}
  \label{eq:def_DN_ela}
  \bLambda : (C^\infty_0&(\cS))^n \to (C^0(\cR))^n,\\
  \bLambda \bf &= \bs(\bu^{\bf})\n |_{\cR},
\end{align}
where $\bu^{\bf}$ solves \eqref{eq:direct_problem_ela} with the Dirichlet boundary condition $\bf$. Our main result is stated as below.

\begin{theorem}
\label{main-Elasticity}
\it Let $\Omega\subset \mathbb{R}^n$ be a bounded connected open set with smooth boundary and $\Gamma$ be an open subset of $\p\Omega$. Let $O_1,O_2\subset\subset \Omega$ be (possibly empty) open subsets with smooth boundary, and $\varphi_1\in C^\infty(\p O_1)$, $\varphi_2\in C^\infty(\p O_2)$.
Assume that $\Omega\setminus \overline{O_1},\Omega\setminus \overline{O_2}$ are connected. 
Let $\bLambda_1$, $\bLambda_2$ be the Dirichlet-to-Neumann map defined in \eqref{eq:def_DN_ela} with obstacles $O_1,\varphi_1$ and $O_2,\varphi_2$, respectively.
If $\bLambda_1=\bLambda_2$, then $O_1=O_2$ and $\varphi_1 = \varphi_2$ .
\end{theorem}
\subsection{Literature review and related results}
The classical inverse obstacle problem seeks to detect unknown inclusions through their interaction with probing fields, such as acoustic, electromagnetic, or gravitational waves \cite{Isakov09,CK18}. Analytically, the physical properties of the inclusion dictate fixed boundary conditions imposed on its surface: sound-soft obstacles require Dirichlet conditions, impenetrable impedance media use Robin conditions, and perfect conductors enforce the vanishing of the tangential trace \cite{CK19}. The foundational uniqueness theorem for obstacle identification was originally established by Schiffer within the context of acoustic scattering \cite{LP67}. Following this, a substantial body of literature has been dedicated to establishing uniqueness and stability for obstacles with fixed boundary conditions using finite boundary measurements \cite{AR01,BV99,Kress04,Rundell08,Bacchelli09}. Beyond pure shape reconstruction, an alternative paradigm couples the unknown obstacle with the surrounding medium's coefficients, recovering the internal structure from the Dirichlet-to-Neumann map. This approach, fundamentally linked to Calder\'on's problem, was pioneered by Isakov \cite{ISakov88} and subsequently advanced by Ikehata \cite{Ikehata98} via a constructive proof. Exploiting this structural equivalence, classical tools from inverse conductivity problems—namely, complex geometrical optics and Runge approximation (see \cite{FSU25} and references therein)—were successfully adapted to recover obstacles from the Dirichlet-to-Neumann map \cite{IINSU07,UW07,UWW09,NUW05}. Despite this extensive body of work, the study concerning inverse problems for obstacles subject to free boundary conditions remains limited; to the author's knowledge, the only existing paper on this topic is \cite{HLLOZ25}.

The Signorini condition provides the standard mathematical formulation for the frictionless contact of an elastic body \cite{Ciarlet88,KO88,CHR23}. Beyond its origins in solid mechanics, the scalar version of the problem has attracted considerable attention in recent decades due to its connections with various other areas of mathematics; see \cite{CV10,CDM16,ROS16,ROS18,Fernandez22} and references therein. The central mathematical complexity of the Signorini problem stems from the presence of a free boundary, which arises naturally since the actual contact region is unknown a priori. Motivated by these analytical challenges, much of the existing literature has focused on establishing the regularity of this free boundary \cite{ACS08,DS16,KPS15,FJ21,FRS20,FS18,FR21,FT23}. The presence of a free boundary also severely complicates the associated inverse problem. Unlike standard formulations with fixed boundary conditions, where the inhomogeneous obstacle function can be determined via unique continuation from outer boundary, the free boundary effectively shields the obstacle, making it generally unobservable from the outer boundary. As illustrated in Example \ref{example:lap_2} of Section \ref{sec:examples}, the obstacle function generally cannot be seen from the Dirichlet-to-Neumann map restricted to any bounded set.

Our proof is based on the approximate boundary controllability of the elliptic mixed Dirichlet-Neumann boundary value problem, classically known as the Zaremba problem since it was first considered by Zaremba in \cite{Zaremba10}. A well-known technical challenge associated with this boundary condition is that solutions typically exhibit singular behavior, failing to be globally smooth even when the domain and prescribed data are perfectly smooth \cite[Problem 3.2.15]{Kenig94}. Consequently, the regularity theory for the Zaremba problem has generated a substantial body of literature; we refer the reader to \cite{AK82,L86,L89,S97,LCB08} and the references therein. In the context of controllability, a Carleman estimate for the Zaremba problem was derived in \cite{CR14} to prove null controllability for its parabolic counterpart. This framework was subsequently adapted in \cite{ZOZ18} to obtain a conditional exact controllability result. To the best of our knowledge, the present work serves as the first application of boundary controllability for the elliptic Zaremba problem to the field of inverse problems.

\subsection{Outline of the paper} The paper is organized as follows. Section \ref{sec:mixed_lap} is dedicated to the regularity theory of a mixed boundary value problem, prescribing Dirichlet and Neumann conditions on $\p \Omega$ and $\p O$, respectively. Furthermore, we establish a high-order approximate controllability result by exploiting the unique continuation property of the corresponding dual system. In Section \ref{sec:lap_main}, we present the proof of Theorem \ref{main-Laplace}, which proceeds in two main steps: we first prove that the Dirichlet-to-Neumann map uniquely determines the shape of the obstacle, and then show that it also determines the inhomogeneous obstacle function. Section \ref{sec:mixed_ela} extends this analysis to the vectorial case, investigating the regularity and boundary controllability for the Lam\'e system. Building upon these foundations, Section \ref{sec:ela_main} establishes Theorem \ref{main-Elasticity} through a method analogous to that of the scalar setting. Finally, Section \ref{sec:examples} provides several counterexamples illustrating the necessity of using the full Dirichlet-to-Neumann map to ensure uniqueness.

\subsection{Acknowledgements} The author would like to thank Lauri Oksanen for helpful discussions. The author was supported by the Finnish Ministry of Education and Culture’s Pilot for Doctoral Programmes (Pilot project Mathematics of Sensing, Imaging and Modelling).
\section{Mixed boundary value problem}
\label{sec:mixed_lap}
This section is devoted to showing the well-posedness and the boundary approximate controllability of the following mixed boundary value problem.
\begin{equation}
  \label{eq:mixed_lap}
  \begin{cases}
    \Delta v = \psi,\text{ in }\Omega \setminus \oO,\\
    v|_{\p\Omega} = f,\ \p_{\bn} v|_{\p O} = g.
  \end{cases}
\end{equation}


In general, singularities may occur at the interface between regions where Dirichlet and Neumann conditions are imposed separately(see \cite[Section 8.2.2]{Grisvard_Ell}). Consequently, the solution to the Zaremba problem admits only local regularity, making it impossible to obtain global smoothness regardless of how regular the data may be.
However, as $\p O$ and $\p \Omega$ are disjoint in our case, the singular solutions will not arise and we can expect the regularity comparable to the pure boundary problems.
\begin{lemma}
  \label{lm:mixed_lap_s_positive}
  Let $s\geq 0$ and $\psi\in H^s(\Omega\setminus\oO)$, $f\in H^{s+\frac{3}{2}}(\p\Omega)$, $g\in H^{s+\frac{1}{2}}(\p O)$. Then \eqref{eq:mixed_lap} admits a unique solution $u\in H^{s+2}(\Omega\setminus \oO)$. Moreover, we have the estimate
  \begin{equation}
    \norm{u}_{H^{s+2}(\Omega\setminus \oO)}\lesssim \norm{\psi}_{H^{s}(\Omega\setminus\oO)} + \norm{f}_{H^{s+\frac{3}{2}}(\p \Omega)} + \norm{g}_{H^{s+\frac{1}{2}}(\p O)}.
  \end{equation}
\end{lemma}
\begin{proof}
  According to the trace theorem, there exists a function $\phi\in H^{s+2}(\Omega\setminus\oO)$ such that $\phi|_{\p\Omega} = f$ and $\p_{\bn} \phi|_{\p O} = g$, and 
  \begin{equation}
    \norm{\phi}_{H^{s+2}(\Omega\setminus \oO)}\lesssim \norm{f}_{H^{s+\frac{3}{2}}(\p \Omega)} + \norm{g}_{H^{s+\frac{1}{2}}(\p O)}.
  \end{equation}
  Then $w : = u-\phi$ solves 
  \begin{equation}
    \label{eq:test_mix_lap}
    \begin{cases}
      \Delta w = \psi - \Delta \phi,\text{ in }\Omega \setminus \overline{O},\\
      w|_{\p\Omega} = 0,\ \p_{\bn} w|_{\p O} = 0.
    \end{cases}
  \end{equation}
  We define the space 
  \begin{equation}
    H^1_D := \{u\in H^1(\Omega\setminus\oO)\mid u|_{\p\Omega} = 0\}.
  \end{equation}
  Then variational problem for \eqref{eq:test_mix_lap} amounts to finding $w\in H^1_D(\Omega\setminus\oO)$ such that 
  \begin{equation}
    \label{eq:test_lap_var}
    \int_{\Omega\setminus\oO} \nabla w\cdot \nabla v\ dx =  - \int_{\Omega\setminus \oO} (\psi - \Delta \phi) v\ dx,\text{ for all }v\in H^1_D(\Omega\setminus\oO).
  \end{equation}
  According to Lax-Milgram theorem (see e.g. \cite[Corollary 5.8]{Brezis}), \eqref{eq:test_mix_lap} admits a unique solution $w\in H^1_D(\Omega\setminus \oO)$. 
  Taking $v = w$ in \eqref{eq:test_lap_var} yields
  \begin{equation}
    \norm{\nabla w}_{L^2(\Omega\setminus \oO)}^2 \leq \norm{\psi-\Delta \phi}_{L^2(\Omega\setminus \oO)} \norm{w}_{L^2(\Omega\setminus \oO)}
  \end{equation}
  Moreover, due to the Poincar\'e inequality (see e.g. \cite[Corollary 4.5.3]{Ziemer89}), we have
  \begin{equation}
    \norm{w}_{H^1(\Omega\setminus \oO)}^2 \lesssim \norm{\nabla w}_{L^2(\Omega\setminus \oO)}^2\leq \norm{\psi-\Delta \phi}_{L^2(\Omega\setminus \oO)} \norm{w}_{H^1(\Omega\setminus \oO)}.
  \end{equation}
  It follows that 
  \begin{equation}
    \norm{w}_{H^1(\Omega\setminus \oO)} \lesssim \norm{\psi}_{L^2(\Omega\setminus \oO)} + \norm{\Delta \phi}_{L^2(\Omega\setminus \oO)}.
  \end{equation}
  Since $O\subset\subset \Omega$, there exists smooth domains $U_1$, $U_2$ and $U_2$ such that 
  \begin{equation}
    O\subset\subset U_1\subset\subset U_2\subset \subset \Omega.
  \end{equation}
  Then $w$ solves a local elliptic Dirichlet system in $\Omega \setminus \overline{U_1}$ and a local elliptic Neumann system in $U_2\setminus\oO$.
  Thanks to \cite[Theorem 4.18]{Mclean2000}, there holds
  \begin{align}
    &\norm{w}_{H^{s+2}(\Omega\setminus\oO)}\lesssim \norm{\psi-\Delta \phi}_{H^s(\Omega\setminus\overline{U_1})}+\norm{\psi-\Delta \phi}_{H^s(U_2\setminus \oO)} + \norm{w}_{H^1(\Omega\setminus \oO)}\\
    &\lesssim \norm{\psi}_{H^s(\Omega\setminus\oO)} + \norm{\Delta\phi}_{H^s(\Omega\setminus\oO)} \leq \norm{\psi}_{H^s(\Omega\setminus\oO)} + \norm{\phi}_{H^{s+2}(\Omega\setminus\oO)}\\
    &\lesssim \norm{\psi}_{H^{s}(\Omega\setminus\oO)} + \norm{f}_{H^{s+\frac{3}{2}}(\p \Omega)} + \norm{g}_{H^{s+\frac{1}{2}}(\p O)}.
  \end{align}
\end{proof}
To establish the solvability for \eqref{eq:mixed_lap} with distributional boundary value, we will need the following concepts.

Let $U\subset \R^n$ be a bounded domain with smooth boundary $\p U$. 
For $s\geq 0$, we define the norm
\begin{equation}
  \trinorm{u}_{-s} : = \norm{u}_{(H^{s}(U))^\ast} + \norm{u}_{H^{-s-\frac{1}{2}}(\p U)} + \norm{\p_{\bn} u}_{H^{-s-\frac{3}{2}}(\p U)}.
\end{equation}
And we denote the completion of $C^\infty(\overline{U})$ with norm $\trinorm{\cdot}_{-s}$ by $\tH^{-s}(U)$. In view of \cite[Theorem 6.1.1]{Roitberg96}, the norm $\trinorm{\cdot}_{-s}$ is equivalent to the graph norm of $\Delta$, that is 
\begin{equation}
  \label{eq:norm_equivalence}
  \trinorm{u}_{-s}\lesssim \norm{u}_{(H^{s}(U))^\ast}+\norm{\Delta u}_{(H^{s+2}(U))^\ast}\lesssim \trinorm{u}_{-s}.
\end{equation}

Next we show the existence and uniqueness for the solution to the system 
  \begin{equation}
    \label{eq:mixed_lap_distributional}
    \begin{cases}
      \Delta u = 0,\text{ in }\Omega\setminus \oO,\\
      u|_{\p\Omega} = f,\ \p_{\bn}u|_{\p O} = g.
    \end{cases}
  \end{equation}
with distributional boundary value $f$ and $g$.
\begin{lemma}
  \label{lm:well_pose_mix_lap}
  Let $s\geq 0$ and $f\in H^{-s-\frac{1}{2}}(\p\Omega)$, $g\in H^{-s-\frac{3}{2}}(\p O)$. Then \eqref{eq:mixed_lap_distributional} admits a unique solution $u\in \tH^{-s}(\Omega\setminus \oO)$.
\end{lemma}
\begin{proof}
  Let $\phi\in H^s(\Omega\setminus\oO)$, we consider the equation
  \begin{equation}
    \label{eq:test_mix_lap}
    \begin{cases}
      \Delta w = \phi,\text{ in }\Omega \setminus \overline{O},\\
      w|_{\p\Omega} = 0,\ \p_{\bn} w|_{\p O} = 0.
    \end{cases}
  \end{equation}
  According to Lemma \ref{lm:mixed_lap_s_positive}, we have $w\in H^{s+2}(\Omega\setminus\oO)$. Moreover, by virtue of the trace theorem, there holds
  \begin{equation}
    \label{eq:trace_w}
    \norm{w}_{H^{s+\frac{3}{2}}(\p O)}+\norm{\p_{\n} w}_{H^{s+\frac{1}{2}}(\p\Omega)}\lesssim \norm{w}_{H^{s+2}(\Omega\setminus \oO)}\lesssim \norm{\phi}_{H^s(\Omega\setminus\oO)}.
  \end{equation}
  Then we define a linear functional $L$ on $H^{s}(\Omega\setminus\oO)$ as 
  \begin{equation}
    L(\phi) : = \inner{f,\p_{\n} w}_{H^{-s-\frac{1}{2}}(\p\Omega), H^{s+\frac{1}{2}}(\p\Omega)} - \inner{g,w}_{H^{-s-\frac{3}{2}}(\p O), H^{s+\frac{3}{2}}(\p O)},
  \end{equation}
  where $w$ solves \eqref{eq:test_mix_lap} associated with $\phi$. Since the solution for \eqref{eq:test_mix_lap} is unique, the functional $L$ is well-defined. Moreover, in view of \eqref{eq:trace_w}, there holds
  \begin{align}
    L(\phi) &\leq \norm{f}_{H^{-s-\frac{1}{2}}(\p\Omega)} \norm{\p_{\n} w}_{H^{s+\frac{1}{2}}(\p\Omega)} + \norm{g}_{H^{-s-\frac{3}{2}}(\p O)} \norm{w}_{H^{s+\frac{3}{2}}(\p O)}\\
    &\lesssim (\norm{f}_{H^{-s-\frac{1}{2}}(\p\Omega)}+\norm{g}_{H^{-s-\frac{3}{2}}(\p O)}) \norm{\phi}_{H^s(\Omega\setminus\oO)}.
  \end{align}
  Thus $L$ is also continuous, and we can write 
  \begin{equation}
    \label{eq:L_to_u}
    L(\phi) = \inner{u,\phi}_{(H^{s}(\Omega\setminus\oO))^\ast, H^s(\Omega\setminus \oO)}
  \end{equation}
  with $u \in (H^{s}(\Omega\setminus\oO))^\ast$. Next we verify that $u$ solves \eqref{eq:mixed_lap}. Notice that for any test function $\psi\in \mathcal{D}(\Omega\setminus \oO)$, we have 
  \begin{align}
    &\inner{\Delta u, \psi}_{\mathcal{D}(\Omega\setminus \oO)^\ast , \mathcal{D}(\Omega\setminus \oO)} = \inner{u,\Delta \psi}_{\mathcal{D}(\Omega\setminus \oO)^\ast , \mathcal{D}(\Omega\setminus \oO)} = L(\Delta\psi)\\
    & = \inner{f,\p_{\bn} \psi}_{(C^\infty_0(\p\Omega))^\ast , C^\infty_0(\p\Omega)} - \inner{g,\psi}_{(C^\infty_0(\p O))^\ast , C^\infty_0(\p O)} = 0.
  \end{align}
  Hence $\Delta u = 0$ in $\Omega\setminus \oO$ in the weak sense (see \cite[Definition 6.1.1]{Roitberg96}). Thanks to \cite[Theorem 6.1.1]{Roitberg96}, we have $u\in \tH^{-s}(\Omega\setminus \oO)$ and thus $u|_{\p \Omega}\in H^{-s-\frac{1}{2}}(\p\Omega)$, $\p_{\bn} u|_{\p O}\in H^{-s-\frac{3}{2}}(\p O)$. By virtue of the trace theorem (see e.g. \cite[Chapter 1, Theorem 9.4]{Lions}), for any $\xi\in H^{s+\frac{3}{2}}(\p O)$ and $\zeta\in H^{s+\frac{1}{2}}(\p\Omega)$, there exists $v_\xi, v_\zeta\in H^{s+2}(\Omega\setminus \oO)$ such that 
  \begin{equation}
    v_\xi|_{\p\Omega} = 0,\ v_\xi|_{\p O} = \xi,\ \p_{\n} v_\xi|_{\p\Omega} = 0,\ \p_{\bn} v_\xi|_{\p O} = 0,
  \end{equation}
  and 
  \begin{equation}
    v_\zeta|_{\p\Omega} = 0,\ v_\zeta|_{\p O} = 0,\ \p_{\n} v_\zeta|_{\p\Omega} = \zeta, \ \p_{\bn} v_\zeta|_{\p O} = 0.
  \end{equation}
  Using Lemma \ref{lm:Gauss_Green_lap} and \eqref{eq:L_to_u}, we have 
  \begin{align}
  \inner{g , \xi}_{H^{-s-\frac{3}{2}}(\p O), H^{s+\frac{3}{2}}(\p O)} &= - \inner{u,\Delta v_\xi}_{H^{-s}(\Omega\setminus \oO), H^s(\Omega\setminus \oO)}\\
  & =  \inner{\p_{\bn} u , v_\xi}_{H^{-s-\frac{3}{2}}(\p O), H^{s+\frac{3}{2}}(\p O)},
  \end{align}
  and 
  \begin{align}
    \inner{f , \zeta}_{H^{-s-\frac{1}{2}}(\p \Omega), H^{s+\frac{1}{2}}(\p \Omega)} &= - \inner{u,\Delta v_\zeta}_{H^{-s}(\Omega\setminus \oO), H^s(\Omega\setminus \oO)}\\
    & =  \inner{u , \p_{\n} v_\zeta}_{H^{-s-\frac{1}{2}}(\p \Omega), H^{s+\frac{1}{2}}(\p \Omega)}.
  \end{align}
  Then we can conclude that $u|_{\p\Omega} = f$ in $H^{-s-\frac{1}{2}}(\p \Omega)$ and $\p_{\bn} u|_{\p O} = g$ in $H^{-s-\frac{3}{2}}(\p O)$ as $\xi\in H^{s+\frac{3}{2}}(\p O)$ and $\zeta\in H^{s+\frac{1}{2}}(\p\Omega)$ are arbitrary. 

  Finally, since \eqref{eq:mixed_lap} is a linear equation, the uniqueness for its solution follows immediately from Lemma \ref{lm:mixed_lap_s_positive}.
\end{proof}
Next we consider the following mixed boundary value problem.
\begin{equation}
  \label{eq:mixed_lap_0}
  \begin{cases}
    \Delta v = 0,\text{ in }\Omega \setminus \oO,\\
    v|_{\p\Omega} = f,\ \p_{\bn} v|_{\p O} = 0.
  \end{cases}
\end{equation}

Let us define the space
\begin{equation}
  \label{eq:def_B_lap}
  \iE : = \{v^f|_{\p O} \mid f\in C^\infty_0(\cS), v^{f} \text{ solves \eqref{eq:mixed_lap_0}}\}.
\end{equation}
Now we are ready to show the following high order approximate controllability from $\p\Omega$ to $\p O$.
\begin{proposition}
  \label{prop:approxi_control_lap}
  The space $\iE$ defined in \eqref{eq:def_B_lap} is dense in $H^{s+\frac{3}{2}}(\p O)$ with respect to $H^{s+\frac{3}{2}}(\p O)$ topology for all $s\geq 0$. 
\end{proposition}
\begin{proof}
  We argue by contradiction. Suppose that $\iE$ is not dense in $H^{s+\frac{3}{2}}(\p O)$. According to the Hahn-Banach theorem (see e.g. \cite[Corollary 1.8]{Brezis}), there exists a non-zero $h\in H^{-k-\frac{3}{2}}(\p O)$ such that for all $g\in \iE$, there holds
  \begin{equation}
    \inner{h,g}_{H^{-s-\frac{3}{2}}(\p O), H^{s+\frac{3}{2}}(\p O)}=0.
  \end{equation}
  Consider the equation
  \begin{equation}
    \begin{cases}
      \Delta w = 0,\text{ in }\Omega\setminus \oO,\\
      w|_{\p\Omega} = 0,\ \p_{\bn} w|_{\p O} = h.
    \end{cases}
  \end{equation}
  According to Lemma \ref{lm:well_pose_mix_lap}, we have $w\in \tH^{-s}(\Omega\setminus \oO)$. Moreover, by virtue of Lemma \ref{lm:Gauss_Green_lap}, there holds
  \begin{align}
    0=&\inner{\Delta w,v^f}_{(H^{s+2}(\Omega\setminus\oO))^\ast, H^{s+2}(\Omega\setminus\oO)}  - \inner{w,\Delta v^f}_{(H^{s}(\Omega\setminus\oO))^\ast, H^s(\Omega\setminus\oO)} \\
    &=\inner{\p_{\n} w, v^f}_{H^{-s-\frac{3}{2}}(\p \Omega), H^{s+\frac{3}{2}}(\p \Omega)} +\inner{\p_{\bn} w, v^f}_{H^{-s-\frac{3}{2}}(\p O), H^{s+\frac{3}{2}}(\p O)}\\
    & - \inner{w,\p_{\n} v^f}_{H^{-s-\frac{1}{2}}(\p \Omega), H^{s+\frac{1}{2}}(\p \Omega)} - \inner{w,\p_{\bn} v^f}_{H^{-s-\frac{1}{2}}(\p O), H^{s+\frac{1}{2}}(\p O)}\\
    & =\inner{\p_{\n} w, f}_{H^{-k-\frac{1}{2}}(\p\Omega), H^{k+\frac{1}{2}}(\p\Omega)}.
  \end{align}
  for all $f\in C^\infty_0(\cS)$. Thus $\p_{\n} w|_{\cS} = 0$ in the distribution sense. Since $O\subset \subset \Omega$, we can find an open set $U$ with smooth boundary $\p U$ such that $O\subset \subset U\subset\subset \Omega$. Thanks to the interior regularity for elliptic equations (see e.g. \cite[Theorem 17.1.3, Lemma 17.1.5]{Hormander3}), we have $w$ is smooth in a neighborhood of $\p U$. Then according to the regularity for elliptic equation with smooth Dirichlet boundary value in $\Omega\setminus \overline{U}$, we can obtain $w\in C^\infty(\overline{\Omega}\setminus U)$. Notice that $w|_\cS = 0 $ and $\p_{\n} w|_\cS = 0$, then we can conclude that $w = 0$ in $\Omega\setminus \overline{U}$ by unique continuation. Therefore, by virtue of unique continuation from an open set (see e.g. \cite[Theorem 17.2.6]{Hormander3}), it follows that $w = 0$ in $\Omega\setminus \overline{O}$, and hence $h = \p_{\bn} w|_{\p O} = 0$. This leads to a contradiction.
\end{proof}

\section{Proof of Theorem \ref{main-Laplace}}
\label{sec:lap_main}
To prove Theorem \ref{main-Laplace}, we will need the following lemma.
\begin{lemma}
  \label{lm:lap_variational}
  Let $U$ be open and a set of finite perimeter and $\nu$ be the outward measure-theoretic unit normal vector to $U$. Let $\psi \in L^\infty (\p^\ast U)$.
  Suppose $u\in C^1(\overline{U})\cap H^2(U)$ satisfying $\Delta u = 0$ in $U$ and 
  \begin{equation}
    u\geq \psi,\ \p_{\nu} u \geq 0,\ (u-\psi) \p_{\nu} u = 0,\text{ a.e. on }\p^\ast U.
  \end{equation}
  Then $u$ is a constant function in $U$.
\end{lemma}
\begin{proof}
  We decompose $\p^\ast U = \Gamma_1\cup \Gamma_2$, where $\Gamma_1$ and $\Gamma_2$ are measurable sets defined as
  \begin{equation}
    \Gamma_1 : = \{x\in \p^\ast U \mid u(x) = \psi(x)\},\ \Gamma_2 : = \{x\in \p^\ast U \mid u(x) \neq \psi(x)\}.
  \end{equation}
  Since $(u-\psi) \p_{\bn} u = 0$ a.e. on $\p^\ast U$, we have $\p_{\bn} u = 0$ a.e. on $\Gamma_2$. Let $M$ be a constant function in $U$ such that $M>\sup \psi$. Then we notice that $(M-\psi)\p_{\bn} u\geq 0$ a.e. on $\Gamma_1$.
  As $u\in C^1(\overline{U})\cap H^2(U)$, by the Sobolev extension (see e.g. \cite[Section 5.4]{evans}), we can extend it to a neighborhood of $U$ in the same regularity class.
  Using the Gauss-Green formula in \cite[Proposition 6.4]{CCT19}, there holds
  \begin{align}
    - \int_U \abs{\nabla u}^2 &= \int_U \nabla u\cdot \nabla(M-u) = \int_{\p^\ast U} (M-u) \p_{\bn} u \ d\cH^{n-1} - \int_U (M-u) \Delta u\\
    &= \int_{\Gamma_1} (M-\psi)\p_{\bn} u \ d\cH^{n-1}\geq 0.
  \end{align}
  Therefore, we have $\nabla u = 0$ a.e. in $U$. It follows immediately that $u$ is a constant function in $U$ since $u\in C^1(\overline{U})$. 
\end{proof}
Above lemma shows that when the Signorini condition is prescribed over the entire boundary, the Laplace equation admits only constant solutions. This rigidity arises because the unilateral constraint forces the net energy flux across the boundary to vanish. In the following theorem, we leverage this property to prove that the shape of the obstacle is uniquely determined by the Dirichlet-to-Neumann map $\Lambda$.
\begin{theorem}
  \it Let $\Omega\subset \mathbb{R}^n$ be a bounded connected open set with smooth boundary. Let $O_1,O_2\subset\subset \Omega$ be (possibly empty) open subsets with smooth boundary, and $\varphi_1\in C^\infty(\p O_1)$, $\varphi_2\in C^\infty(\p O_2)$.
  Assume that $\Omega\setminus \overline{O_1},\Omega\setminus \overline{O_2}$ are connected. 
  Let $\Lambda_1$, $\Lambda_2$ be the Dirichlet-to-Neumann map defined in \eqref{eq:def_DN} with obstacles $O_1,\varphi_1$ and $O_2,\varphi_2$, respectively.
  If $\Lambda_1=\Lambda_2$, then $O_1=O_2$.
\end{theorem}
\begin{proof}
  To get a contradiction, we suppose $O_1\neq O_2$. Then without loss of generality, we assume that $O_1\not\subset O_2$, and we define the following sets analogous to \cite[Section 2]{HLLOZ25}:
\begin{align} 
  G:= \textrm{the connected component of } \Omega\setminus (\overline{O_1\cup O_2}) \textrm{ such that } \p \Omega\subset \p G, \label{def-G0} \\
  \cV:= \textrm{a connected component of } (\Omega\setminus \overline{G})\setminus \overline{O_2} \textrm{ satisfying } \p\cV \cap \p G\neq \emptyset. \label{def-Vset}
\end{align}
\begin{figure}[ht]
  \includegraphics[width = 8cm]{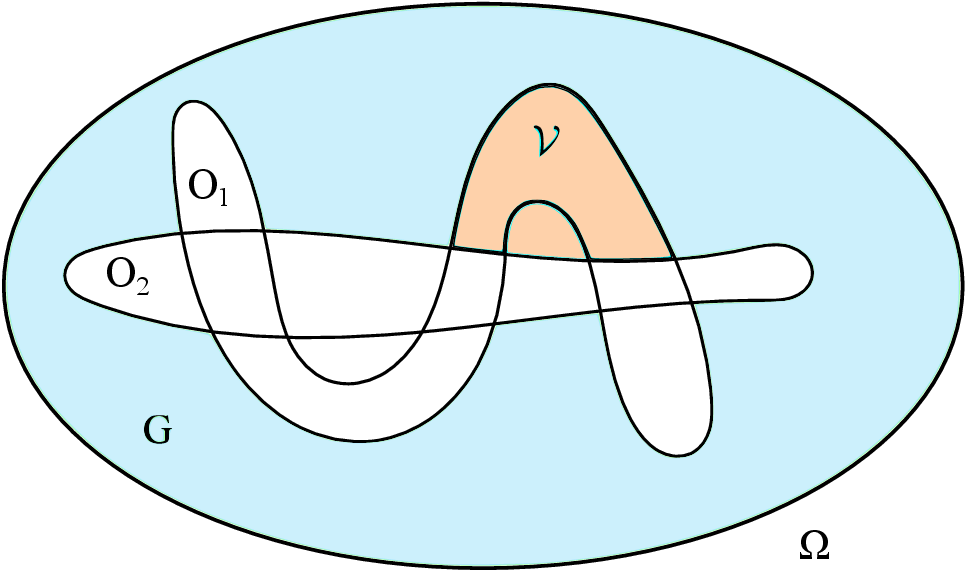}
  \caption{An illustration for set $G$ and $\cV$.}
\end{figure}
  Let $u_i^f$ solve \eqref{eq:direct_problem_lap} with Dirichlet condition $f$ on $\p\Omega$ and obstacle $O_i$, $\varphi_i$, for $i=1,2$, respectively. In view of \cite[Lemma 2.2]{FJ77} and \cite[Theorem 2.3.5]{GP09}, there holds $u_i^f\in H^2(\Omega\setminus \oO_i)\cap C^{1,\frac{1}{2}}(\overline{\Omega}\setminus O_i)$ for $i=1,2$, respectively. Let us denote the outward unit normal to $\Omega\setminus \oO_i$ on $\p O_i$ by $\nu_i$ for $i=1,2$, respectively.
  Let $\nu_\cV$ be the outward measure-theoretic unit normal to $\cV$ on the reduced boundary $\p^\ast \cV$.
  By virtue of \cite[Proposition 3.3]{HLLOZ25}, we have $\nu_\cV = \nu_2$ a.e. on $\p^\ast \cV\cap \p O_2$ and $\nu_\cV = - \nu_1$ a.e. on $\p^\ast \cV\setminus \p O_1$.
  Then the Signorini boundary condition for $u_2$ on $\p O_2$ reads
  \begin{equation}
    \label{eq:thm1_eq1}
    u_2^f\geq \varphi_2, \;\; \partial_{\nu}u_2^f\geq 0, \;\; (u_2^f-\varphi_2)\partial_{\nu} u_2^f =0 \textrm{ a.e. on } \p^\ast \cV\cap \p O_2,
  \end{equation}  
  for all $f\in C^\infty_0(\cS)$. Let $v^f_1$ solve
  \begin{equation}
    \begin{cases}
    \Delta v = 0,\text{ in }\Omega\setminus O_1,\\
    v|_{\p \Omega} = f, \p_{\nu_1} v|_{\p O_1} = 0.
    \end{cases}
  \end{equation}
  We claim that there exists a non-constant function $g\in C^\infty_0(\cS)$, such that $v^g_1 > \varphi_1$ on $\p O_1$. Indeed, let $\delta>0$ and $\phi\in C^\infty(\p O_1)$ such that $\phi>\varphi_1 + \delta$.
  According to Proposition \ref{prop:approxi_control_lap} and the Sobolev embedding theorem, there exists $h\in C^\infty_0 (\cS)$ such that $\norm{v^h|_{\p O} - \phi}_{C^0(\p O_1)}\leq \frac{\delta}{2}$. Thus $v^h_1|_{\p O}\geq \phi-\frac{\delta}{2}>\varphi_1$. and $v^f_1$ solves \eqref{eq:direct_problem_lap}. If $h$ is not a constant function, we set $g=h$. Otherwise, let $\eta\in C^\infty_0(\cS)$ be a non-constant function, then by Lemma \ref{lm:well_pose_mix_lap} and the trace theorem, there exists a small enough $\varepsilon>0$ such that 
  \begin{equation}
    \norm{v^{\varepsilon \eta}_1}_{C^0(\p O_1)} = \varepsilon \norm{v^{\eta}_1}_{C^0(\p O_1)} \leq \frac{\delta}{2}.
  \end{equation}
  Then we set $g = h+\varepsilon\eta$ which is non-constant and there holds
  \begin{equation}
    v^{g}|_{\p O}=v^{h+\varepsilon \eta}|_{\p O} \geq \phi-\delta >\varphi_1.
  \end{equation}
  This justifies our claim. Notice that $v_1^g$ solves \eqref{eq:direct_problem_lap} with the Dirichlet boundary condition $g$ on $\p\Omega$ and the obstacle $O_1$, $\varphi_1$. According to the uniqueness for the solution to \eqref{eq:direct_problem_lap}, we can conclude that $u_1^g = v_1^g$ in $\Omega\setminus \oO_1$. Due to $\Lambda_1 f = \Lambda_2 f$ and the unique continuation, we have $u_1^f = u_2^f$ in $G$. By the argument used in the proof of \cite[Theorem 1]{HLLOZ25}, we have 
  \begin{equation}
    \label{eq:thm1_eq2}
    u_2^g = u_1^g\geq \varphi_1,\text{ and }\partial_{\nu_1}u_2^g =\partial_{\nu_1}u_1^g = 0 \text{ a.e. on } \partial^*\mathcal{V}\setminus \partial O_2.
  \end{equation}
  We define a function $\psi \in L^\infty(\p^\ast \cV)$ as
  \begin{equation}
    \psi(x) : = 
    \begin{cases}
      \varphi_2(x),\text{ if }x\in \p^\ast \cV\cap \p O_2,\\
      \varphi_1(x),\text{ if }x\in \p^\ast \cV\setminus \p O_2.
    \end{cases}
  \end{equation}
  We note that $\psi$ is well-defined on $\p^\ast \cV$ in view of \cite[Lemma 2.4]{HLLOZ25}. Combining \eqref{eq:thm1_eq1} and \eqref{eq:thm1_eq2} yields
  \begin{equation}
    u_2^g\geq \psi, \;\; \partial_{\nu}u_2^g\geq 0, \;\; (u_2^f-\psi)\partial_{\nu} u_2^g =0 \textrm{ a.e. on } \p^\ast \cV.
  \end{equation}
  According to Lemma \ref{lm:lap_variational}, we have $u_2^g$ is a constant function in $\cV$. By virtue of \cite[Lemma 2.4]{HLLOZ25}, $\cV$ is connected to $\p\Omega$. Applying the unique continuation for $u^g_2$, we have $g$ is a constant on $\p\Omega$. This is a contradiction.
\end{proof}
Assume $\Lambda_1 = \Lambda_2$, now we can write $O = O_1 = O_2$. It only remains to show that $\varphi_1 = \varphi_2$ on $O$. To this end, we will need the uniqueness for the solution to a more complicated mixed boundary value problem. Let $\p O = \Gamma_N\cup \Gamma_S$, where $\Gamma_N$ is an open subsets of $\p O$, and let $\varphi$ be a smooth function on $\Gamma_S$. Then we consider the following mixed boundary value problem
\begin{equation}
  \label{eq:Sig_mix_lap}
  \begin{cases}
  \hfil  \Delta u = 0,\text{ in }\Omega\setminus \oO, \\ 
  \hfil  u|_{\p\Omega} = f,\ \p_{\bn} u|_{\Gamma_N} = 0,\\
  \hfil  u\geq \varphi,\ \p_{\bn} u \geq 0,\ (u-\varphi)\p_{\bn} u = 0,\text{ a.e. on }\Gamma_S. 
  \end{cases}
\end{equation}
\begin{lemma}
  \label{lm:Sig_mix_lap}
  Let $f\in C^\infty(\p\Omega)$. Then there exists a unique $u\in H^1(U)$ that solves \eqref{eq:Sig_mix_lap}.
\end{lemma}
\begin{proof}
  According to the trace theorem, we can find  a function $\phi\in H^2(\Omega\setminus\oO)$ such that 
  \begin{equation}
    \phi|_{\p\Omega} = f,\ \phi|_{\p O} = 0,\ \p_{\bn} \phi|_{\p O} = 0.
  \end{equation}
  Therefore, $w : = u-\phi$ solves 
  \begin{equation}
    \label{eq:Sig_mix_lap_w}
  \begin{cases}
  \hfil  \Delta w = -\Delta \phi,\text{ in }\Omega\setminus \oO, \\ 
  \hfil  w|_{\p\Omega} = 0,\ \p_{\bn} w|_{\Gamma_N} = 0,\\
  \hfil  w\geq \varphi,\ \p_{\bn} w \geq 0,\ (w-\varphi)\p_{\bn} w = 0,\text{ a.e. on }\Gamma_S. 
  \end{cases}
\end{equation}
  Define the set
  \begin{equation}
    V : = \{v\in H^1(U) \mid v = 0 \text{ on }\p\Omega, v\geq \varphi \text{ a.e. on }\Gamma_S\}
  \end{equation}
  and the bilinear form 
  \begin{equation}
    a(w,v): = \int_{\Omega\setminus \oO} \nabla w\cdot \nabla v\ dx,\ w,v\in V.
  \end{equation}
  Due to the Poincare's inequality, $a(\cdot, \cdot)$ is a continuous coercive bilinear form on $V$.
  Then the variational form of \eqref{eq:Sig_mix_lap_w} amounts to finding $w\in V$ such that 
  \begin{equation}
    a(w,v-w) \geq \int_{\Omega\setminus \oO} (v-w)\Delta \phi \ dx
  \end{equation}
  Since $\Delta \phi\in L^2(\Omega\setminus\oO)$, the solvability and the uniqueness for the solution to this variational problem follows immediately from the Stampacchia's theorem (see e.g. \cite[Theorem 5.6]{Brezis}). 
\end{proof}
\begin{theorem}
  \it Let $\Omega\subset \mathbb{R}^n$ be a bounded connected open set with smooth boundary. Let $O\subset\subset \Omega$ be a non-empty open subset with smooth boundary, and $\varphi_1,\varphi_2\in C^\infty(\p O)$.
  Assume that $\Omega\setminus \overline{O}$ is connected. 
  Let $\Lambda_1$, $\Lambda_2$ be the Dirichlet-to-Neumann map defined in \eqref{eq:def_DN} with obstacle functions $\varphi_1$, $\varphi_2$, respectively.
  If $\Lambda_1=\Lambda_2$, then $\varphi_1=\varphi_2$.
\end{theorem}
\begin{proof}
  To get a contradiction, we assume there exists $x\in \p O$ such that $\varphi_1(x)\neq \varphi_2(x)$. Since $\varphi_1$ and $\varphi_2$ are continuous, we can find a small enough $r>0$ such that $\varphi_1\neq \varphi_2$ in the ball $B(x,r)\cap \p O$, where $B(x,r)$ is the ball in $\R^n$ centered at $x$ with radius $r$. To simplify the notation, we denote $B(x,r)\cap \p O$ by $\Gamma_r$.
  Since $\Lambda_1 = \Lambda_2$, we have $u_1^f = u_2^f$ in $\Omega \setminus \oO$ by unique continuation. Therefore, we have 
  \begin{equation}
    (u_2^f-\varphi_1)\p_{\bn} u_2^f = (u_1^f-\varphi_1)\p_{\bn} u_1^f = 0 = (u_2^f-\varphi_2)\p_{\bn} u_2^f,\text{ a.e. on }\Gamma_r
  \end{equation}
  for all $f\in C^\infty_0(\cS)$.
  This implies that 
  \begin{equation}
    (\varphi_2 - \varphi_1)\p_{\bn} u_2^f = 0, ,\text{ a.e. on }\Gamma_r.
  \end{equation}
  Since $\varphi_1\neq \varphi_1$ on $\Gamma_r$ and $u_2^f\in C^1(\overline{\Omega}\setminus O)$, we can deduce that 
  \begin{equation}
    \p_{\bn} u_2^f = 0 \text{ on }\Gamma_r
  \end{equation} 
  for all $f\in C^\infty_0(\cS)$. Thus $u_2^f$ solves 
  \begin{equation}
    \label{eq:thm3_1}
  \begin{cases}
  \hfil  \Delta u = 0,\text{ in }\Omega\setminus \oO, \\ 
  \hfil  u|_{\p\Omega} = f,\ \p_{\bn} u|_{\Gamma_r} = 0,\\
  \hfil  u\geq \varphi_2,\ \p_{\bn} u \geq 0,\ (u-\varphi_2)\p_{\bn} u = 0,\text{ a.e. on }\p O\setminus \Gamma_r. 
  \end{cases}
\end{equation}
Let $\phi\in C_0^\infty(\Gamma_r)$ such that $\phi\geq 0$ and let $y\in \Gamma_r$ be such that $\phi(y)>0$. Let $\delta>0$ and $\psi\in C^\infty(\p O)$ such that $\psi>\varphi_2+\delta$. Then we can find a large enough constant $C\in \R$ such that $\rho: = \psi - C \phi$ satisfies
\begin{equation}
  \rho(y) = \psi(y) - C \phi(y)<\varphi_2(y) - \delta.
\end{equation}
By virtue of Proposition \ref{prop:approxi_control_lap}, there exists $f\in C^\infty_0(\cS)$ such that $\norm{v^f - \rho}_{C^0(\p O)}\leq \delta$, where $v^f$ solves \eqref{eq:mixed_lap_0}. Notice that $\p_{\bn} v^f = 0$ on $\p O$ and 
\begin{equation}
  v^f \geq \rho - \delta = \psi-\delta > \varphi_2,\text{ on }\p O\setminus \Gamma_r,
\end{equation}
since $\phi=0$ on $\p O\setminus \Gamma_r$.
Thus $v^f$ also solves \eqref{eq:thm3_1}. By virtue of Lemma \ref{lm:Sig_mix_lap}, we have $v^f = u_2^f$. Hence, there holds
\begin{equation}
  u_2^f(y) = v^f(y) \leq \rho(y) +\delta < \varphi_2(y).
\end{equation}
This leads to a contradiction to the fact that $u_2^f$ satisfies the Signorini boundary condition on $\p O$.
\end{proof}

\section{Mixed boundary value problem for Lam\'e system}
\label{sec:mixed_ela}
In this section, we extends this analysis in Section \ref{sec:mixed_lap} to the vectorial case and investigate the regularity and boundary controllability for the following mixed boundary value problem for the Lam\'e
system.
\begin{equation}
    \label{eq:mixed_ela}
    \begin{cases}
    \hfil  \div \bs(\bu)=\bpsi, \text{ in }\Omega\setminus \overline{O},\\
    \hfil \bu|_{\p\Omega}=\boldsymbol{f},\ \bs(\bu)\bn|_{\p O} = \bg.
    \end{cases}
  \end{equation} 
To simplify the notations, for any subset $U$ of $\R^n$ and $s\in \R$, we denote $(H^s(U))^n$ by $\bH^s(U)$. And for any $\bu: = (u_1,\cdots, u_n)\in \bH^s(U)$, the norm on $\bH^s(U)$ reads 
\begin{equation}
  \norm{\bu}_{\bH^s(U)} : =\left( \sum_{i=1}^n \norm{u_i}_{H^s(U)}^2\right)^{\frac{1}{2}}.
\end{equation}
Similarly, we denote $(\tH^{-s}(\Omega\setminus \oO))^n$ and $(\mathcal{D}(U))^n$ by $\tbH^{-s}(\Omega\setminus \oO)$ and $\boldsymbol{\mathcal{D}}(U)$, respectively.

Next we show the well-posedness of \eqref{eq:mixed_ela} provided that the source and boundary values are regular enough.
\begin{lemma}
  \label{lm:mixed_ela_s_positive}
  Let $s\geq 0$ and $\bpsi\in \bH^s(\Omega\setminus\oO)$, $\bf\in \bH^{s+\frac{3}{2}}(\p\Omega)$, $\bg\in \bH^{s+\frac{1}{2}}(\p O)$. Then \eqref{eq:mixed_lap} admits a unique solution $\bu\in \bH^{s+2}(\Omega\setminus \oO)$. Moreover, we have the estimate
  \begin{equation}
    \norm{\bu}_{\bH^{s+2}(\Omega\setminus \oO)}\lesssim \norm{\bpsi}_{\bH^{s}(\Omega\setminus\oO)} + \norm{\bf}_{\bH^{s+\frac{3}{2}}(\p \Omega)} + \norm{\bg}_{\bH^{s+\frac{1}{2}}(\p O)}.
  \end{equation}
\end{lemma}
\begin{proof}
  First we consider the homogeneous problem
  \begin{equation}
    \label{eq:mixed_ela_homo}
    \begin{cases}
    \hfil  \div \bs(\bu)=0, \text{ in }\Omega\setminus \overline{O},\\
    \hfil \bu|_{\p\Omega}=0,\ \bs(\bu)\bn|_{\p O} = 0.
    \end{cases}
  \end{equation}
  We define the vector-valued function space
  \begin{equation}
    \bV : = \{\bv\in \bH^1(\Omega\setminus \oO)\mid \bv|_{\p\Omega} = 0\}.
  \end{equation}
  Then we define the bilinear form associated with the Lam\'e system as
  \begin{equation}
    A(\bv,\bw): = \int_{\Omega\setminus \oO} \bs(\bv): \be(\bw)\ dx,\ \bv,\bw\in \bV.
  \end{equation} 
  If $\bv\in \bV$ is a solution of \eqref{eq:mixed_ela_homo}, by virtue of \cite[Theorem 4.4]{Mclean2000}, there holds 
  \begin{align}
    0&=A(\bv,\bv) = \int_{\Omega\setminus \oO} (2\mu \be(\bv)+\lambda\tr(\be(v))\boldsymbol{I}_n): \be(\bv)\ dx\\
    &=\int_{\Omega\setminus \oO} 2\mu \be(\bv):\be(\bv)+\lambda\tr(\be(v))^2 \ dx.
  \end{align}
  Since $\nu$ and $\lambda$ are positive smooth functions, we have $\be(\bv):\be(\bv) = 0$ a.e. in $\Omega\setminus \oO$, and consequently, $\norm{\be(\bv)}_{(L^2(\Omega\setminus\oO))^{n\times n}} = 0$. Moreover, since $\cH^{n-1}(\p\Omega)>0$, applying Korn's inequality (see e.g. \cite[eq(4.10)]{sofonea2012}) on $\bV$ yields
  \begin{equation}
    \norm{\bv}_{\bH^1(\Omega\setminus \oO)}\lesssim \norm{\be(\bv)}_{(L^2(\Omega\setminus\oO))^{n\times n}} = 0.
  \end{equation}
  Therefore, the homogeneous problem \eqref{eq:mixed_ela_homo} has only the trivial solutions. Thanks to \cite[Theorem 4.10]{Mclean2000}, there exists a unique solution $\bu\in \bH^1(\Omega\setminus \oO)$ for \eqref{eq:mixed_lap}. Moreover, there holds
  \begin{align}
    \norm{\bu}_{\bH^1(\Omega\setminus \oO)}&\lesssim \norm{\bpsi}_{(\bH^1(\Omega\setminus \oO))^\ast} + \norm{\bf}_{\bH^{\frac{1}{2}}(\p\Omega)} + \norm{\bg}_{\bH^{-\frac{1}{2}}(\p O)}\\
    &\leq \norm{\bpsi}_{\bH^{s}(\Omega\setminus\oO)} + \norm{\bf}_{\bH^{s+\frac{3}{2}}(\p \Omega)} + \norm{\bg}_{\bH^{s+\frac{1}{2}}(\p O)}.
  \end{align}
  for any $s\geq 0$. Analogous to the proof for Lemma \ref{lm:mixed_lap_s_positive}, since $O\subset\subset \Omega$, we can find smooth domains $U_1$, $U_2$ and $U_2$ such that 
  \begin{equation}
    O\subset\subset U_1\subset\subset U_2\subset \subset \Omega.
  \end{equation}
  Thus $\bu$ solves a local elliptic Dirichlet system in $\Omega \setminus \overline{U_1}$ and a local elliptic Neumann system in $U_2\setminus\oO$.
  Thanks to \cite[Theorem 4.18]{Mclean2000}, there holds
  \begin{align}
    \norm{\bu}_{\bH^{s+2}(\Omega\setminus \oO)}&\lesssim \norm{\bpsi}_{\bH^s(\Omega\setminus \oO)} + \norm{\bf}_{\bH^{s+\frac{3}{2}}(\p\Omega)}  + \norm{\bg}_{\bH^{s+\frac{1}{2}}(\p O)} +  \norm{\bu}_{\bH^{1}(\Omega\setminus \oO)}\\
    &\lesssim \norm{\bpsi}_{\bH^{s}(\Omega\setminus\oO)} + \norm{\bf}_{\bH^{s+\frac{3}{2}}(\p \Omega)} + \norm{\bg}_{\bH^{s+\frac{1}{2}}(\p O)}.
  \end{align}
\end{proof} 
\begin{lemma}
  \label{lm:mixed_ela_s_negative}
  Let $s\geq 0$ and $\bf\in \bH^{-s-\frac{1}{2}}(\p\Omega)$, $\bg\in \bH^{-s-\frac{3}{2}}(\p O)$. Then \eqref{eq:mixed_ela} admits a unique solution $\bu\in \tbH^{-s}(\Omega\setminus \oO)$.
\end{lemma}
\begin{proof}
  Let $\bpsi \in \bH^{s}(\Omega\setminus \oO)$ and $\bw$ solves 
  \begin{equation}
    \label{eq:mixed_ela_0}
    \begin{cases}
      \div \bs(\bw) = \bpsi,\text{ in }\Omega\setminus \oO,\\
      \bw|_{\p\Omega} = 0,\ \bs(\bw)\bn|_{\p O} = 0.
    \end{cases}
  \end{equation}
  According to Lemma \ref{lm:mixed_ela_s_positive}, we have $\bw\in \bH^{s+2}(\Omega\setminus \oO)$. Moreover, the trace theorem reads
  \begin{equation}
    \norm{\bw}_{\bH^{s+\frac{3}{2}}(\p O)} + \norm{\bs(\bw)\bn}_{H^{s+\frac{1}{2}}(\p\Omega)} \lesssim \norm{\bw}_{\bH^{s+2}(\Omega\setminus \oO)}\lesssim \norm{\bpsi}_{\bH^{s}(\Omega\setminus \oO)}.
  \end{equation}
  Then we define the linear function $\boldsymbol{L}$ on $\bH^s(\Omega\setminus \oO)$ 
  \begin{equation}
    \boldsymbol{L}(\bpsi) : = \inner{\bf,\bs(\bw)\bn}_{\bH^{-s-\frac{1}{2}}(\p\Omega),\bH^{s+\frac{1}{2}}(\p\Omega)} - \inner{\bg,\bw}_{\bH^{-s-\frac{3}{2}}(\p O),\bH^{s+\frac{3}{2}}(\p O)},
  \end{equation}
  where $\bw$ solves \eqref{eq:mixed_ela_0} with the source $\bpsi$. In view of Lemma \ref{lm:mixed_lap_s_positive}, the solution for \eqref{eq:mixed_ela_0} is unique, thus $\boldsymbol{L}$ is well defined. Moreover, we have
  \begin{align}
    \abs{\boldsymbol{L}(\bpsi)} & \leq \norm{\bf}_{\bH^{-s-\frac{1}{2}}(\p\Omega)} \norm{\bs(\bw)\bn}_{\bH^{s+\frac{1}{2}}(\p\Omega)} + \norm{\bg}_{\bH^{-s-\frac{3}{2}}(\p O)}\norm{\bw}_{H^{s+\frac{3}{2}}(\p O)}\\
    &\lesssim (\norm{\bf}_{\bH^{-s-\frac{1}{2}}(\p\Omega)} + \norm{\bg}_{\bH^{-s-\frac{3}{2}}(\p O)})\norm{\bpsi}_{\bH^s(\Omega\setminus \oO)}.
  \end{align}
  Thus $\boldsymbol{L}$ is continuous. Hence, we can write
  \begin{equation}
    \boldsymbol{L}(\cdot) = \inner{\bu, \cdot}_{\bH^{-s}(\Omega\setminus \oO),\bH^{s}(\Omega\setminus \oO)}.
  \end{equation}
  Notice that the Lam\'e system with real coefficients is self-adjoint. Thus for any $\bpsi\in (C_0^\infty(\Omega\setminus \oO))^n$, there holds
  \begin{align}
    &\inner{\div \bs(\bu), \bpsi}_{\boldsymbol{\mathcal{D}}(\Omega\setminus \oO)^\ast, \boldsymbol{\mathcal{D}}(\Omega\setminus \oO)} = \inner{\bu, \div\bs(\bpsi)}_{\boldsymbol{\mathcal{D}}(\Omega\setminus \oO)^\ast, \boldsymbol{\mathcal{D}}(\Omega\setminus \oO)}\\
    & = \inner{\bu, \div\bs(\bpsi)}_{\bH^{-s}(\Omega\setminus \oO),\bH^{s}(\Omega\setminus \oO)}= \boldsymbol{L}(\div\bs(\bpsi)) \\
    &= \inner{\bf,\bs(\bpsi)\bn}_{\bH^{-s-\frac{1}{2}}(\p\Omega),\bH^{s+\frac{1}{2}}(\p\Omega)} - \inner{\bg,\bpsi}_{\bH^{-s-\frac{3}{2}}(\p O),\bH^{s+\frac{3}{2}}(\p O)} = 0.
  \end{align}
  Therefore, $\bu$ is a weak generalized solution to $\div\bs(\bu) = 0$ in $\Omega\setminus \oO$ in the sense of \cite[eq.(10.4.6)]{Roitberg96}. Thanks to \cite[Theorem 10.4.2]{Roitberg96}, we can conclude that $\bu\in \tbH^{-s}(\Omega\setminus \oO)$. And thus $\bu|_{\p \Omega}\in \bH^{-s-\frac{1}{2}}(\p\Omega)$ and $\bs(\bu)\bn|_{\p O}\in \bH^{-s-\frac{3}{2}}(\p O)$. It only remains to verify that $\bu|_{\p \Omega} = \bf$ and $\bs(\bu)\bn|_{\p O} = \bg$. By virtue of Lemma \ref{lm:trace_ela}, for any $\bxi\in H^{s+\frac{3}{2}}(\p O)$ and $\bzeta\in H^{s+\frac{1}{2}}(\p\Omega)$, there exists $\bv_\xi, \bv_\zeta\in \bH^{s+2}(\Omega\setminus \oO)$ such that 
  \begin{equation}
    \bv_\xi|_{\p\Omega} = 0,\ \bv_\xi|_{\p O} = \bxi,\ \bs(\bv_\xi)\n|_{\p\Omega} = 0,\ \bs(\bv_\xi)\bn|_{\p O} = 0,
  \end{equation}
  and 
  \begin{equation}
    \bv_\zeta|_{\p\Omega} = 0,\ \bv_\zeta|_{\p O} = 0,\ \bs(\bv_\zeta)\n|_{\p\Omega} = \bzeta, \ \bs(\bv_\zeta)\bn|_{\p O} = 0.
  \end{equation}
  Using Lemma \ref{lm:Green_ela}, we have 
  \begin{align}
  &\inner{\bg , \bxi}_{\bH^{-s-\frac{3}{2}}(\p O), \bH^{s+\frac{3}{2}}(\p O)} = - \inner{\bu,\div\bs(\bv_\xi)}_{(\bH^s(\Omega\setminus \oO))^\ast, \bH^s(\Omega\setminus \oO)}\\
  & =  \inner{\bs(\bu)\bn, \bv_\xi}_{H^{-s-\frac{3}{2}}(\p O), H^{s+\frac{3}{2}}(\p O)} = \inner{\bs(\bu)\bn, \bxi}_{H^{-s-\frac{3}{2}}(\p O), H^{s+\frac{3}{2}}(\p O)},
  \end{align}
  and 
  \begin{align}
    &\inner{\bf , \bzeta}_{\bH^{-s-\frac{1}{2}}(\p \Omega), \bH^{s+\frac{1}{2}}(\p \Omega)} = - \inner{\bu,\div\bs(\bv_\zeta)}_{(\bH^s(\Omega\setminus \oO))^\ast, \bH^s(\Omega\setminus \oO)}\\
    & =  \inner{\bu , \bs(\bv_\zeta)\n}_{\bH^{-s-\frac{1}{2}}(\p \Omega), \bH^{s+\frac{1}{2}}(\p \Omega)}  =  \inner{\bu , \bzeta}_{\bH^{-s-\frac{1}{2}}(\p \Omega), \bH^{s+\frac{1}{2}}(\p \Omega)}.
  \end{align}
  Then we can conclude that $\bu|_{\p\Omega} = \bf$ in $\bH^{-s-\frac{1}{2}}(\p \Omega)$ and $\bs(\bu)\bn|_{\p O} = \bg$ in $\bH^{-s-\frac{3}{2}}(\p O)$ as $\bxi\in \bH^{s+\frac{3}{2}}(\p O)$ and $\bzeta\in \bH^{s+\frac{1}{2}}(\p\Omega)$ are arbitrary.  
  
  Finally, the uniqueness of $\bu$ follows immediately from the fact that \eqref{eq:mixed_ela_homo} only has trivial solution, which is shown in Lemma \ref{lm:mixed_ela_s_positive}.
\end{proof}
Next, for $\bf\in (C_0^\infty(\cS))^n$, we consider the mixed boundary value problem
\begin{equation}
  \label{eq:mixed_ela_control}
  \begin{cases}
    \div\bs(\bu) = 0,\text{ in }\Omega\setminus \oO,\\
    \bu|_{\p \Omega} = \bf,\ \bs(\bu)\bn|_{\p O} = 0.
  \end{cases}
\end{equation}
Then we define the space
\begin{equation}
  \label{eq:def_bE}
  \boldsymbol{E}: = \{\bv^{\bf}|_{\p O} \mid \bf\in (C_0^\infty(\cS))^n,\ \bv^{\bf} \text{ solves \eqref{eq:mixed_ela_control}}\}.
\end{equation}
We have the following higher order approximate controllability result.
\begin{proposition}
  \label{prop:approximate_control_ela}
  Let $s\geq 0$. The space $\boldsymbol{E}$ defined in \eqref{eq:def_bE} is dense in $\bH^{s+\frac{3}{2}}(\p O)$ with respect to $\bH^{s+\frac{3}{2}}(\p O)$ topology.
\end{proposition}
\begin{proof}
  To get a contradiction, we assume that $\boldsymbol{E}$ is not dense in $\bH^{s+\frac{3}{2}}(\p O)$. By virtue of the Hahn-Banach theorem, there must exists a non-trivial $\bh\in \bH^{-s-\frac{3}{2}}(\p O)$ such that 
  \begin{equation}
    \inner{\bh,\bg}_{\bH^{-s-\frac{3}{2}}(\p O),\bH^{s+\frac{3}{2}}(\p O)}, \text{ for all }\bg\in \boldsymbol{E}.
  \end{equation}
  Then we consider the mixed boundary value problem
  \begin{equation}
    \begin{cases}
      \div \bs(\bw) = 0, \text{ in }\Omega\setminus \oO,\\
      \bw|_{\p \Omega} = 0,\ \bs(\bw)\bn|_{\p O} = \bh.
    \end{cases}
  \end{equation}
  In view of Lemma \ref{lm:mixed_ela_s_negative}, above system admits a unique solution $\bw\in \tbH^{-s}(\Omega\setminus\oO)$. By Lemma \ref{lm:Green_ela}, we have
  \begin{align}
    0=&\inner{\div\bs(\bw),\bv^{\bf}}_{(\bH^{s+2}(\Omega\setminus\oO))^\ast, \bH^{s+2}(\Omega\setminus\oO)}  - \inner{\bw,\div \bs(\bv^{\bf})}_{(\bH^{s}(\Omega\setminus\oO))^\ast, \bH^s(\Omega\setminus\oO)} \\
    &=\inner{\bs(\bw)\n, \bv^{\bf}}_{\bH^{-s-\frac{3}{2}}(\p \Omega), \bH^{s+\frac{3}{2}}(\p \Omega)} +\inner{\bs(\bw)\bn, \bv^{\bf}}_{\bH^{-s-\frac{3}{2}}(\p O), \bH^{s+\frac{3}{2}}(\p O)}\\
    & - \inner{\bw,\bs(\bv^{\bf})\n}_{\bH^{-s-\frac{1}{2}}(\p \Omega), \bH^{s+\frac{1}{2}}(\p \Omega)} - \inner{\bw,\bs(\bv^{\bf})\bn}_{\bH^{-s-\frac{1}{2}}(\p O), \bH^{s+\frac{1}{2}}(\p O)}\\
    & =\inner{\bs(\bw)\n, \bf}_{\bH^{-k-\frac{1}{2}}(\p\Omega), \bH^{k+\frac{1}{2}}(\p\Omega)}.
  \end{align}
  Since $\bf\in (C_0^\infty(\cS))^n$ is arbitrary, there holds $\bs(\bw)\n|_{\cS} = 0$ in the distributional sense. According to the interior regularity for the elliptic system (see e.g. \cite[Theorem 8.3]{Giaquinta83}), we have $\bw\in (C^\infty_{\text{loc}}(\Omega\setminus \oO))$. Let $U$ be an open domain with smooth boundary $\p U$ and $O\subset\subset U\subset\subset \Omega$. Then $\bw|_{\p U}\in C^\infty(\p U)$ and $\bw$ satisfies
  \begin{equation}
    \begin{cases}
      \div \bs(\bw) = 0,\text{ in }\Omega\setminus U,\\
      \bw|_{\p U}\in C^\infty(\p U),\ \bw|_{\p \Omega} \in C^\infty(\p \Omega).
    \end{cases}
  \end{equation}
  Using Lemma \ref{lm:mixed_ela_s_positive}, we can conclude that $\bw\in C^\infty(\overline{\Omega}\setminus U)$. Since $\bw|_{\cS} = \sigma(\bw)\n|_{\cS} = 0$, by unique continuation for local Cauchy data of Lam\'e system (see e.g. \cite[Corollary 2.2]{EHMR21}), we have $\bw = 0$ in $\Omega\setminus \overline{U}$. Applying the unique continuation from open subset for Lam\'e system (see e.g. \cite[Theorem 2.3]{UW09}), there holds $\bw = 0$ in $\Omega\setminus \oO$ and thus $\bh = \bs(\bw)\bn|_{\p O} = 0$. This leads to a contradiction. 
\end{proof}
\section{Proof of Theorem \ref{main-Elasticity}}
\label{sec:ela_main}
In this section, we prove Theorem \ref{main-Elasticity} with an approach analogous to Section \ref{sec:lap_main}. In contrast to Lemma \ref{lm:lap_variational}, where a scalar version of Signorini condition restricts the Laplace equation to constant solutions, the Lam\'e system exhibits greater complexity as it lacks such simple trivial solutions. Consequently, we must rely on the variational inequality established in the following lemma.
\begin{lemma}
  \label{lm:ela_variational}
  Let $U$ be open and a set of finite perimeter and $\nu$ be the outward measure-theoretic unit normal vector to $U$. Let $\psi \in L^\infty (\p^\ast U)$.
  Suppose $\bu\in C^1(\overline{U})^n \cap \bH^2(U)$ satisfy $\div\bs(\bu) = 0$ in $U$ and 
  \begin{equation}
    \bs(\bu)_\tau=0,\ \bu_\nu \leq \psi,\ \bs(\bu)_\nu \leq 0,\ (\bu_\nu-\psi)\bs(\bu)_\nu=0,\text{ a.e. on }\p^\ast U.
  \end{equation}
  Then for any $\bv\in K_{\psi,U}$, where 
  \begin{equation}
    \label{eq:def_K}
    K_{\psi,U} : = \{\bv \in C^1(\overline{U})^n \cap \bH^2(U) \mid \bv_{\nu}\leq \psi \text{ a.e. on }\p^\ast U \},
  \end{equation}
  we have 
  \begin{equation}
    \int_{U} \bs(\bu): \be(\bv - \bu) \geq 0.
  \end{equation}
\end{lemma}
\begin{proof}
  We decompose $\p^\ast U = \Gamma_1\cup \Gamma_2$, where $\Gamma_1$ and $\Gamma_2$ are measurable sets defined as
  \begin{equation}
    \Gamma_1 : = \{x\in \p^\ast U \mid \bu_\nu(x) = \psi(x)\},\ \Gamma_2 : = \{x\in \p^\ast U \mid \bu_\nu(x) \neq \psi(x)\}.
  \end{equation}
  Since $(\bu_\nu-\psi) \bs(\bu)_\nu = 0$ and $\bs(\bu)_\nu\leq 0$ a.e. on $\p^\ast U$, we have $\bs(\bu)_\nu = 0$ a.e. on $\Gamma_2$ and $(\bv_\nu - \psi)\bs(\bu)_\nu\geq 0$ a.e. on $\Gamma_1$ for all $\bv\in K_{\psi,U}$. Using the generalized Gauss-Green formula in \cite[Lemma 5.2]{HLLOZ25}, there holds
  \begin{align}
    &\int_U \bs(\bu):\be(\bv - \bu)  = \int_{\p^\ast U} \bs(\bu)\bn \cdot (\bv-\bu)\ d\cH^{n-1} - \int_{U} \div\bs(\bu)\cdot (\bv-\bu)\\
    & = \int_{\p^\ast U} \bs(\bu)_\nu (\bv-\bu)_\nu\ d\cH^{n-1} = \int_{\Gamma_1} \bs(\bu)_\nu (\bv-\psi)_\nu\ d\cH^{n-1}\geq 0.
  \end{align}
\end{proof}
We denote the set of constant $n\times n$ skew-symmetric matrices by $\mathbb{S}^{n}$, and we define the set
\begin{equation}
  \fR : =\{\bf\in (C^\infty(\p\Omega))^n\mid \bf = A\boldsymbol{x}|_{\p\Omega} + \boldsymbol{c}, A\in \mathbb{S}^n, \boldsymbol{c}\in \R^n\}.
\end{equation}
Next we show that the Dirichlet-to-Neumann map $\bLambda$ uniquely determines the shape of the obstacle with inhomogeneous obstacle function.
\begin{theorem}
  \label{thm:5}
  \it Let $\Omega\subset \mathbb{R}^n$ be a bounded connected open set with smooth boundary. Let $O_1,O_2\subset\subset \Omega$ be (possibly empty) open subsets with smooth boundary, and $\varphi_1\in C^\infty(\p O_1)$, $\varphi_2\in C^\infty(\p O_2)$.
  Assume that $\Omega\setminus \overline{O_1},\Omega\setminus \overline{O_2}$ are connected. 
  Let $\bLambda_1$, $\bLambda_2$ be the Dirichlet-to-Neumann map defined in \eqref{eq:def_DN_ela} with obstacles $O_1,\varphi_1$ and $O_2,\varphi_2$, respectively.
  If $\bLambda_1=\bLambda_2$, then $O_1=O_2$.
\end{theorem}
\begin{proof}
  To get a contradiction, we assume that $O_1\not\subset O_2$ without loss of generality. Let $G$ and $\cV$ be the sets defined in \eqref{def-G0} and \eqref{def-Vset}, respectively.
  Let $\bu_i^{\bf}$ solve \eqref{eq:direct_problem_ela} with Dirichlet condition $\bf$ on $\p\Omega$ and obstacle $O_i$, $\varphi_i$, for $i=1,2$, respectively.
  It was known due to \cite[Theorem 2.2]{KD81} and \cite[Theorem 1]{RS22} that $\bu_i^{\bf}\in \bH^2(\Omega\setminus \oO_i)\cap (C^{1,\frac{1}{2}}(\overline{\Omega}\setminus O_i))^n$ for $i=1,2$, respectively.
  Let us denote the outward unit normal to $\Omega\setminus \oO_i$ on $\p O_i$ by $\nu_i$ for $i=1,2$, respectively.
  Let $\nu_\cV$ be the outward measure-theoretic unit normal to $\cV$ on the reduced boundary $\p^\ast \cV$.
  By virtue of \cite[Proposition 3.3]{HLLOZ25}, we have $\nu_\cV = \nu_2$ a.e. on $\p^\ast \cV\cap \p O_2$ and $\nu_\cV = - \nu_1$ a.e. on $\p^\ast \cV\setminus \p O_1$.
  Then the Signorini boundary condition for $\bu_2^{\bf}$ on $\p O_2$ reads
  \begin{equation}
    \label{eq:thm5_eq1}
    \bs(\bu_2^{\bf})_\tau=0,\ (\bu_2^{\bf})_{\nu_2}\leq \varphi_2,\ \bs(\bu_2^{\bf})_{\nu_2}\leq 0,\ ((\bu_2^{\bf})_{\nu_2}-\varphi_2)\bs(\bu_2^{\bf})_{\nu_2} =0
  \end{equation}  
  a.e. on $\p^\ast \cV\cap \p O_2$ for all $f\in C^\infty_0(\cS)$. Let $\bv^{\bf}_1$ solve
  \begin{equation}
    \begin{cases}
    \div\bs(\bv) = 0,\text{ in }\Omega\setminus \oO_1,\\
    \bv|_{\p \Omega} = \bf,\  \bs(\bv)\nu_1|_{\p O_1} = 0.
    \end{cases}
  \end{equation}
  Let $\delta>0$ and $M$ be a constant such that $M>\varphi_1+\delta$. According to Proposition \ref{prop:approximate_control_ela} and the Sobolev embedding theorem, there exists $\bg\in C^\infty_0(\cS)^n$ such that 
  \begin{equation}
    \norm{\bv^{\bg}_1 - M \nu_1}_{C^0(\p O)^n}\leq \frac{\delta}{2}.
  \end{equation}
  Thus 
  \begin{equation}
    \label{eq:thm5_1}
    (\bv^{\bg}_1)_{\nu_1} = \bv^{\bg}_1 \cdot \nu_1 = M + (\bv^{\bg}\cdot \nu_1 - M )\geq M-\frac{\delta}{2} > \varphi_1.
  \end{equation}
  Let $\bdeta\in C_0^\infty(\cS)$ such that $\bdeta\neq \fR$. By virtue of Lemma \ref{lm:mixed_ela_s_positive} and the trace theorem, there exists a small enough $\varepsilon>0$ such that
  \begin{equation}
    \norm{\bv^{\varepsilon \bdeta}_1}_{C^0(\p O)^n} = \varepsilon\norm{\bv^{\bdeta}_1}_{C^0(\p O)^n}\leq \frac{\delta}{2}.
  \end{equation}
  Write $\bh : = \bg+\varepsilon\bdeta$, then we have 
  \begin{equation}
    \label{eq:thm5_2}
    (\bv^{\bh}_1)_{\nu_1} = \bv^{\bg}_1 \cdot \nu_1 + \varepsilon\bv^{\bdeta}_1\cdot \nu_1\geq M - \delta >\varphi_1.
  \end{equation}
  Notice that for any $\bf\in C^\infty(\p \cS)^n$, there holds 
  \begin{equation}
    \label{eq:thm5_3}
    \bs(\bv_1^{\bf})_{\nu_1} = \bs(\bv_1^{\bf})\nu_1\cdot \nu_1 = 0,\ \bs(\bv_1^{\bf})_\tau = \bs(\bv^{\bf}_1)\nu_1 - \bs(\bv_1^{\bf})_{\nu_1} \nu_1 = 0.
  \end{equation}
  Hence \eqref{eq:thm5_1} and \eqref{eq:thm5_2} imply that $\bv_1^{\bg}$, $\bv^{\bh}_1$ solve \eqref{eq:direct_problem_ela} with the obstacle $O_1$, $\varphi_1$ and the Dirichlet boundary condition $\bg$, $\bh$, respectively. Analogous to the argument used in the proof of \cite[Theorem 2]{HLLOZ25}, due to the uniqueness for the solution to \eqref{eq:direct_problem_ela}, we can conclude that $\bv^{\bg}_1 = \bu_1^{\bg}$ and $\bv^{\bh}_1 = \bu_1^{\bh}$. Since $\bLambda_1 \bg = \bLambda_1 \bg$ and $\bLambda_1 \bh = \bLambda_1 \bh$, there holds $\bu_1^{\bg} = \bu_1^{\bg}$ and $\bu_1^{\bh} = \bu_1^{\bh}$ in $G$. Moreover, we have 
  \begin{equation}
    \label{eq:thm5_eq2}
    \bu_2^{\bg}=\bu_1^{\bg} >\varphi_1,\ \bs(\bu_2^{\bg})_{\tau} = \bs(\bu_1^{\bg})_{\tau} = 0,\ \bs(\bu_2^{\bg})_{\nu_1} = \bs(\bu_1^{\bg})_{\nu_1} = 0, 
  \end{equation}
  and 
  \begin{equation}
    \label{eq:thm5_eq3}
    \bu_2^{\bh}=\bu_1^{\bh} >\varphi_1,\ \bs(\bu_2^{\bh})_{\tau} = \bs(\bu_1^{\bh})_{\tau} = 0,\ \bs(\bu_2^{\bh})_{\nu_1} = \bs(\bu_1^{\bh})_{\nu_1} = 0
  \end{equation}
  a.e. on $\p^\ast \cV \setminus \p O_2$. Next we define a function $\psi \in L^\infty(\p^\ast \cV)$ as
  \begin{equation}
    \psi(x) : = 
    \begin{cases}
      \varphi_2(x),\text{ if }x\in \p^\ast \cV\cap \p O_2,\\
      \varphi_1(x),\text{ if }x\in \p^\ast \cV\setminus \p O_2.
    \end{cases}
  \end{equation}
  And we note that $\psi$ is well-defined on $\p^\ast \cV$ in view of \cite[Lemma 2.4]{HLLOZ25}. Now combining \eqref{eq:thm5_eq1}, \eqref{eq:thm5_eq2} and \eqref{eq:thm5_eq3} we have
  \begin{equation}
    \bs(\bu_2^{\bg})_\tau=0,\ (\bu_2^{\bg})_{\nu_{\cV}} \leq \psi,\ \bs(\bu_2^{\bg})_{\nu_{\cV}} \leq 0,\ ((\bu_2^{\bg})_{\nu_{\cV}}-\psi)\bs(\bu_2^{\bg})_{\nu_{\cV}}=0
  \end{equation}
  and 
  \begin{equation}
    \bs(\bu_2^{\bh})_\tau=0,\ (\bu_2^{\bh})_{\nu_{\cV}} \leq \psi,\ \bs(\bu_2^{\bh})_{\nu_{\cV}} \leq 0,\ ((\bu_2^{\bh})_{\nu_{\cV}}-\psi)\bs(\bu_2^{\bh})_{\nu_{\cV}}=0,
  \end{equation}
  a.e. on $\p^\ast \cV$. Thus $\bu_2^{\bg},\bu_2^{\bh} \in K_{\psi,\cV}$ where $K_{\psi,\cV}$ is defined in \eqref{eq:def_K}. Applying Lemma \ref{lm:ela_variational} for both $\bu_2^{\bg}$ and $\bu_2^{\bh}$ yields
  \begin{equation}
    \label{eq:thm5_eq4}
    \int_{\cV} \bs(\bu_2^{\bg}) : \be(\bu_2^{\bh} - \bu_2^{\bg}) = \int_{\cV} 2\mu \be(\bu_2^{\bg}) : \be(\bu_2^{\bh} - \bu_2^{\bg}) +\lambda \tr(\be(\bu_2^{\bg}))\tr(\be(\bu_2^{\bh} - \bu_2^{\bg}))\geq 0,
  \end{equation}
  and 
  \begin{align}
    \label{eq:thm5_eq5}
    &- \int_{\cV} \bs(\bu_2^{\bh}) : \be(\bu_2^{\bg} - \bu_2^{\bh}) = \int_{\cV} \bs(\bu_2^{\bh}) : \be(\bu_2^{\bh} - \bu_2^{\bg})\\
    & = \int_{\cV} 2\mu \be(\bu_2^{\bh}) : \be(\bu_2^{\bh} - \bu_2^{\bg}) +\lambda \tr(\be(\bu_2^{\bh}))\tr(\be(\bu_2^{\bh} - \bu_2^{\bg}))\leq 0
  \end{align}
  Then \eqref{eq:thm5_eq5} minus \eqref{eq:thm5_eq4} gives
  \begin{align}
    \int_{\cV} 2\mu \be(\bu_2^{\bh}-\bu_2^{\bg}) : \be(\bu_2^{\bh} - \bu_2^{\bg}) +\lambda \tr(\be(\bu_2^{\bh}-\bu_2^{\bg}))^2 \leq 0.
  \end{align}
  Since coefficients $\mu$ and $\lambda$ are positive, we have $\be(\bu_2^{\bh}-\bu_2^{\bg}) = 0$ in $\cV$. In view of \cite[Lemma A.3]{HLLOZ25}, there holds $\bu_2^{\bh}-\bu_2^{\bg} = A\boldsymbol{x}+ \boldsymbol{c}$ in $\cV$ where $A\in \mathbb{S}^n$ and $\boldsymbol{c}\in \R^n$. By \cite[Lemma 2.4]{HLLOZ25}, $\cV$ is connected to $\p \Omega$. Using the unique continuation for $\bu_2^{\bh}-\bu_2^{\bg}$, we have 
  \begin{equation}
    \bu_2^{\bh}|_{\p\Omega}-\bu_2^{\bg}|_{\p\Omega}= (A\boldsymbol{x}+ \boldsymbol{c})|_{\p\Omega} = \bh-\bg = \varepsilon\bdeta \in \fR.
  \end{equation}
  This leads to a contradiction.
\end{proof}
In view of Theorem \ref{thm:5}, we can write $O = O_1 = O_2$ provided that $\bLambda_1 = \bLambda_2$ and thus the obstacle functions $\varphi_1$ and $\varphi_2$ are defined on the same domain. Next we show that the obstacle function can also be determined from the Dirichlet-to-Neumann map $\bLambda$.
\begin{theorem}
  \it Let $\Omega\subset \mathbb{R}^n$ be a bounded connected open set with smooth boundary. Let $O\subset\subset \Omega$ be a non-empty open subsets with smooth boundary, and $\varphi_1,\varphi_2\in C^\infty(\p O)$.
  Assume that $\Omega\setminus \overline{O}$ is connected. 
  Let $\bLambda_1$, $\bLambda_2$ be the Dirichlet-to-Neumann map defined in \eqref{eq:def_DN_ela} with obstacle functions $\varphi_1$, $\varphi_2$, respectively.
  If $\bLambda_1=\bLambda_2$, then $\varphi_1=\varphi_2$.
\end{theorem}
\begin{proof}
  To get a contradiction, we assume that $\varphi_1 \neq \varphi_2$, then the open set 
  \begin{equation}
    \Gamma_s : = \{x\in \p O\mid \varphi_1(x)\neq \varphi_2(x)\}
  \end{equation}
  is non-empty,
  Since $\bLambda_1 = \bLambda_2$, for any $\bf\in (C_0^\infty(\cS))^n$, we have $\bu_1^{\bf} = \bu_2^{\bf}$ in $\Omega \setminus \oO$ by unique continuation. Therefore, we have 
  \begin{equation}
    ((\bu_2^{\bf})_\nu-\varphi_1)(\bs(\bu_2^{\bf}))_\nu = ((\bu_1^{\bf})_\nu-\varphi_1)(\bs(\bu_1^{\bf}))_\nu = 0 = ((\bu_2^{\bf})_\nu-\varphi_2)(\bs(\bu_2^{\bf}))_\nu
  \end{equation}
  on $\Gamma_s$ for all $\bf\in (C_0^\infty(\cS))^n$.
  This implies that 
  \begin{equation}
    (\varphi_2 - \varphi_1)(\bs(\bu_2^{\bf}))_\nu = 0, ,\text{ on }\Gamma_s.
  \end{equation}
  Since $\varphi_1\neq \varphi_1$ on $\Gamma_r$ and $\bu_2^{\bf}\in (C^1(\overline{\Omega}\setminus O))^n$, we can deduce that 
  \begin{equation}
    (\bs(\bu_2^{\bf}))_\nu = 0 \text{ on }\Gamma_s
  \end{equation} 
  for all $\bf\in (C_0^\infty(\cS))^n$. Moreover, the Signorini boundary condition for $\bu_2^{\bf}$ reads $(\bs(\bu_2^{\bf}))_\tau = 0$ on $\Gamma_s\subset \p O$. Consequently, we have 
  \begin{equation}
    \bs(\bu_2^{\bf})\bn = (\bs(\bu_2^{\bf}))_\tau + (\bs(\bu_2^{\bf}))_\nu \bn = 0
  \end{equation}
  on $\Gamma_s$. Thus $\bu_2^{\bf}$ solves 
  \begin{equation}
    \label{eq:thm6_1}
  \begin{cases}
  \hfil  \div\bs(\bu) = 0,\text{ in }\Omega\setminus \oO, \\ 
  \hfil  \bu|_{\p\Omega} = \bf,\ \bs(\bu)\bn|_{\Gamma_s} = 0,\\
  \hfil  \bs(\bu)_\tau=0,\ \bu_\nu \leq \varphi,\ \bs(\bu)_\nu \leq 0,\ (\bu_\nu-\varphi)\bs(\bu)_\nu=0,\text{ on }\p O\setminus \Gamma_s. 
  \end{cases}
\end{equation}
In view of \cite[Theorem 5.3]{sofonea2012}, the system \eqref{eq:thm6_1} admits a unique weak solution in $\bu\in\bH^1(\Omega\setminus \oO)$.
Let $\bphi\in C_0^\infty(\Gamma_s)^n$ such that there is a point $y\in \Gamma_s$ satisfying $\bphi(y)\cdot \bn(y)>0$. Let $\delta>0$ and $\bpsi\in C^\infty(\p O)$ such that $\bpsi\cdot \bn<\varphi_2-\delta$. Then we can find a large enough constant $C\in \R$, such that for $\brho: = \bpsi + C \bphi$, there holds
\begin{equation}
  \brho(y) \cdot \bn = \bpsi(y)\cdot \bn + C \bphi(y)\cdot \bn>\varphi_2(y) + \delta.
\end{equation}
By virtue of Proposition \ref{prop:approximate_control_ela}, there exists $\bf\in (C_0^\infty(\cS))^n$ such that $\norm{\bv^{\bf} - \brho}_{C^0(\p O)^n}\leq \delta$, where $\bv^{\bf}$ solves \eqref{eq:mixed_ela_0}. Notice that $\bs(\bv^{\bf})\bn = 0$ on $\p O$ and
\begin{equation}
  \bv^{\bf}\cdot \bn \leq \brho\cdot \bn - \delta = \bpsi \cdot \bn + \delta < \varphi_2,\text{ on }\p O\setminus \Gamma_s
\end{equation}
since $\bphi = 0$ in $\p O\setminus \Gamma_s$.
Thus $\bv^{\bf}$ also solves \eqref{eq:thm6_1}. By the uniqueness of the solution to \eqref{eq:thm6_1}, we have $\bv^{\bf} = \bu_2^{\bf}$. Hence, there holds
\begin{equation}
  \bu_2^{\bf}(y) \cdot \bn = \bv^{\bf}\cdot \bn \geq \brho(y)\cdot \bn -\delta > \varphi_2(y).
\end{equation}
This leads to a contradiction to the fact that $\bu_2^{\bf}$ satisfies the Signorini boundary condition on $\p O$.
\end{proof}
\section{The Insufficiency of Boundary Measurements: counterexamples}
\label{sec:examples}
It is shown in \cite{HLLOZ25} that a single boundary measurement $\Lambda(f)|_{\cR}$ uniquely determines the shape of a homogeneous Signorini obstacle provided that $f$ is not a constant function. Our first example stands in contrast to this result and demonstrates that an inhomogeneous obstacle fundamentally complicates the inverse problem. 
For simplicity, we only present examples for the scalar case; their elastic counterparts can be constructed analogously. To preclude the insufficiencies arising from partial data, we shall assume that $\cS = \cR = \p \Omega$ throughout this section.
\begin{example}
  \label{example_lap_1}
  Let $z\in \Omega$ such that $\Omega$ is not a ball centered at $z$. 
  Let $u_{z,n}$ be the fundamental solution to Laplacian in $\R^n$. Thats is,
  \begin{equation}
    u_{z,n}(x): = \begin{cases}
      -\frac{1}{2\pi}\log \abs{x-z},\text{ if } n = 2,\\
      \frac{1}{n(n-2)\alpha_n}\frac{1}{\abs{x-z}^{n-2}},\text{ if } n > 2,
    \end{cases}
  \end{equation}
  where $\alpha_n$ is the volume of the unit ball in $\R^n$. Notice that $\nabla u_{z,n}(x) = \frac{-1}{n\alpha_n}\frac{x-z}{\abs{x-z}^{n}}$ and $\Delta u_{z,n} = 0$ in $\R^n\setminus \{z\}$. Let $O\subset \R^n$ be any open convex set such that $z\in O$. Then for any $y\in \p O$, the hyperplane that tangential to $\p O$ at $y$ reads
  \begin{equation}
    P : = \{x\in \R^n \mid x\cdot \bn(y) + b = 0\},
  \end{equation}
  where $\bn(y)$ is the unit outward normal to $\p O$ at $y$.
  Due to the convexity of $O$, we have $x\cdot \bn(y) + b < 0$ for all $x\in O$. In particular, we have $z\cdot \bn(y) + b<0$ since $z\in \Omega$, and thus $(y-z)\cdot\bn(y)>0$. Since the outward normal to $\Omega\setminus\oO$ and $O$ on $\p O$ possess opposite directions, there holds
  \begin{equation}
    \p_{\nu} u_{z,n}(y) = - \nabla u_{z,n}(y) \cdot \bn(y) = \frac{1}{n\alpha_n}\frac{(y-z)\cdot \bn(y)}{\abs{y}^{n}}> 0.
  \end{equation}
  Since $y\in \p O$ is arbitrary, we have $\p_{\nu} u_{z,n} >0$ on $\p O$.
  Therefore, for any convex open set $O\subset \subset \Omega$ such that $z\in O$, $u_{z,n}$ solves \eqref{eq:direct_problem_lap} with $f = u_{z,n}|_{\p\Omega}$ and $\varphi = u_{z,n}|_{\p O}$. Since $\Omega$ is not a ball centered at $z$, $f$ is not a constant function on $\p\Omega$. 
  Therefore the single measurement $\p_{\n} u^f|_{\p\Omega}$ is not able to uniquely determine the shape of obstacle $O$.   
\end{example}
Classical inverse obstacle problems typically prescribe fixed boundary conditions on the surface of scatterer. For example, the mathematical formulation for a sound-soft obstacle with an inhomogeneous Dirichlet condition is given by:
\begin{equation}
\begin{cases}
\Delta u^f = 0, \text{ in } \Omega \setminus \oO,\\
u^f|_{\p \Omega} = f,\ u^f|_{\p O} = \varphi.
\end{cases}
\end{equation}
It is a well known (see, e.g., \cite[Theorem 2]{CK18}) that the shape of obstacle $O$ is uniquely determined by a single pair of Cauchy data $(f, \p_{\n} u^f|_{\p\Omega})$. Once the shape surface $\p O$ is known, the inhomogeneous obstacle profile $\varphi$ is readily recovered via unique continuation.
In contrast, the following example illustrates that restricting the Dirichlet-to-Neumann map $\Lambda$ to a bounded set of boundary source is strictly insufficient to determine the obstacle function $\varphi$ in our setting, primarily because the active contact zone may be empty. Consequently, $\varphi$ remains indeterminate, even when provided with infinitely many linearly independent measurement pairs $(f, \Lambda f)$.
\begin{example}
  \label{example:lap_2}
  For any fixed $N>0$ and $s> \frac{n - 1}{2}$. Suppose that $\Lambda_1 f = \Lambda_2 f$ whenever $\norm{f}_{H^{s}(\p \Omega)} \leq N$. Let $v^f$ solve the mixed boundary value problem \eqref{eq:mixed_lap_0} with the Dirichlet boundary condition $f$ on $\p\Omega$. According to Lemma \ref{lm:mixed_lap_s_positive} and the Sobolev embedding theorem, there holds $\norm{v^f}|_{C^0(\p O)}\leq C N$ for some constant $C$. Therefore, $v^f$ also solves \eqref{eq:direct_problem_lap} associated with the obstacle function $\varphi_1,\varphi_2$ provided that $\varphi_1, \varphi_2 < - C N$. And we have $\Lambda_1 f = \p_{\n} v^f|_{\p \Omega} = \Lambda_2 f$ due to the uniqueness of the solution to \eqref{eq:direct_problem_lap}.
\end{example}  
\nopagebreak

\section{Appendix: Auxiliary Lemmas}
Let $U\subset \R^n$ be a bounded domain with smooth boundary $\p U$, then we have the following lemmas.
\begin{lemma}
  \label{lm:Gauss_Green_lap}
  Let $s\geq 0$ and $u\in \tH^{-s}(U)$, $v\in H^{s+2}(U)$, then there holds
  \begin{align}
    \label{eq:Green_lap}
    &\inner{\Delta u,v}_{(H^{s+2}(U))^\ast, H^{s+2}(U)}  - \inner{u,\Delta v}_{(H^{s}(U))^\ast, H^s(U)} \\
    &=\inner{\p_{\bn} u, v}_{H^{-s-\frac{3}{2}}(\p U), H^{s+\frac{3}{2}}(\p U)} - \inner{u,\p_{\bn} v}_{H^{-s-\frac{1}{2}}(\p U), H^{s+\frac{1}{2}}(\p U)}.
  \end{align}
\end{lemma}
\begin{proof}
  By definition, $C^\infty(\overline{U})$ is dense in $\tH^{-s}(U)$ with respect to the norm $\trinorm{\cdot}_{-s}$. Then we can find sequences $\{\phi_n\}_{n=1}^\infty$ and $\{\psi_n\}_{n=1}^\infty$ in $C^\infty(\overline{U})$ such that as $n\to \infty$, $\phi_n \to u$ in $\tH^{-s}(U)$ and $\psi_n \to v$ in $H^{s+2}(U)$. According to \eqref{eq:norm_equivalence}, there holds
  \begin{equation}
    \lim_{n\to\infty}\norm{\Delta u - \Delta \phi_n}_{(H^{s+2}(U))^\ast} = 0.
  \end{equation}
  Therefore, we have 
  \begin{align}
    \label{eq:Green_lap_1}
    &\lim_{n\to \infty}\int_U   \Delta \phi_n \psi_n - \phi_n \Delta \psi_n\, dx \\
    & = \lim_{n\to \infty}  \inner{\Delta \phi_n,\psi_n}_{(H^{s+2}(U))^\ast, H^{s+2}(U)} - \lim_{n\to \infty} \inner{\phi_n,\Delta \psi_n}_{(H^{s}(U))^\ast, H^{s}(U)}\\
    & = \inner{\Delta u,v}_{(H^{s+2}(U))^\ast, H^{s+2}(U)}  - \inner{u,\Delta v}_{H^{-s}(U), H^s_0(U)}.
  \end{align}
  By virtue of the trace theorem for $H^{s+2}(U)$ (see e.g. \cite[Theorem 1.5.1.2]{Grisvard_Ell}), there holds
  \begin{equation}
    \lim_{n\to\infty}\norm{v-\psi_n}_{H^{s+\frac{3}{2}}(\p U)} = 0,\ \lim_{n\to\infty}\norm{\p_{\bn} v-\p_{\bn} \psi_n}_{H^{s+\frac{1}{2}}(\p U)} = 0.
  \end{equation}
  Then we can obtain
  \begin{align}
    \label{eq:Green_lap_2}
    &\inner{\p_{\bn} u, v}_{H^{-s-\frac{3}{2}}(\p U), H^{s+\frac{3}{2}}(\p U)} - \inner{u,\p_{\bn} v}_{H^{-s-\frac{1}{2}}(\p U), H^{s+\frac{1}{2}}(\p U)}\\
    &= \lim_{n\to \infty}\left( \inner{\p_{\bn} \phi_n, \psi_n}_{H^{-s-\frac{3}{2}}(\p U), H^{s+\frac{3}{2}}(\p U)}-\inner{\phi_n,\p_{\bn} \psi_n}_{H^{-s-\frac{1}{2}}(\p U), H^{s+\frac{1}{2}}(\p U)}\right)\\\
    &=\lim_{n\to \infty} \int_{\p U}\p_{\bn} \phi_n \psi_n - \phi_n \p_{\bn} \psi_n dx.
  \end{align}
  Applying the Gauss-Green identity for smooth functions, we have 
  \begin{equation}
    \int_U \phi_n \Delta \psi_n - \Delta \phi_n \psi_n\, dx = \int_{\p U}\p_{\bn} \phi_n \psi_n - \phi_n \p_{\bn} \psi_n dx
  \end{equation}
  for all $n$. Thus \eqref{eq:Green_lap} follows immediately from passing to the limit $n\to\infty$ in \eqref{eq:Green_lap_1} and \eqref{eq:Green_lap_2}.
\end{proof}
\begin{lemma}
  \label{lm:trace_ela}
  Let $s\geq 0$ and $\bf\in \bH^{s+\frac{3}{2}}(\p U)$, $\bg\in \bH^{s+\frac{1}{2}}(\p U)$. Then there exists a vector function $\bu\in \bH^{s+2}(U)$ such that $\bu|_{\p U} = \bf$ and $\bs(\bu)\bn|_{\p U} = \bg$.
\end{lemma}
\begin{proof}
  We write $\bf = (f_1,\cdots,f_n)$ with $f_i\in H^{s+\frac{3}{2}}(\p U)$ for $1\leq i\leq n$. According to the trace theorem for scalar functions, there exists $v_i \in H^{s+2}(U)$ such that $v_i|_{\p U} = f_i$ for $1\leq i\leq n$. Write $\bv = (v_1,\cdots,v_n)\in \bH^{s+2}(U)$ and $\widetilde{\bg} = \bg - \bs(\bv)\bn|_{\p U} = (\widetilde{g}_1,\cdots,\widetilde{g}_n)\in \bH^{s+\frac{1}{2}}(\p U)$.
  
  We claim that there exists $\bw: = (w_1,\cdots, w_n)\in \bH^{s+2}(U)$ such that $\bw|_{\p U} = 0$ and $\bs(\bw)\bn|_{\p U} = \widetilde{\bg}$. Indeed, by virtue of the trace theorem, there exists $\bw\in H^{s+2}(U)$ such that $\bw|_{\p U} = 0$. Let $y\in \p U$ be arbitrary. Since $w_i|_{\p U} = 0$, for any vector $\xi_y$ that tangential to the boundary $\p U$ at the point $y$ and $1\leq i\leq n$, there holds $\nabla w_i \cdot \xi_y = 0$. Therefore, for all $1\leq i\leq n$, $\nabla w_i$ is normal to $\p U$ and we have $\nabla w_i = \p_{\bn} w_i \bn$. Write $\bn = (\nu_1,\cdots, \nu_n)$, we have $\sum_{j=1}^n \nu_j^2 = 1$ since $\bn$ is a unit vector. Thus the $i-$th component of the system $\bs(\bw)\bn = \widetilde{\bg}$ reads 
  \begin{align}
    \label{eq:lm_trace_ela_1}
    (\bs(\bw)\bn)_i &= \mu \sum_{j=1}^n \left(\p_{\bn} w_i \nu_j\nu_j + \p_{\bn}w_j \nu_i\nu_j \right) +\lambda \sum_{j=1}^n \p_{\bn}w_j \nu_i\nu_j \\
    &=\mu \p_{\bn}w_i +(\mu +\lambda)\nu_i\sum_{j=1}^\infty \p_{\nu}w_j \nu_j = \widetilde{g}_i.
  \end{align}
  Multiplying $\nu_i$ to the both sides and summing over $i$ from $1$ to $n$ yields
  \begin{equation}
    \mu\sum_{i=1}^n\nu_i \p_{\bn}w_i + (\mu+\lambda)\sum_{j=1}^n \nu_j\p_{\bn}w_j = \sum_{i=1}^n \widetilde{g}_i\nu_i.
  \end{equation} 
  Therefore, we have 
  \begin{equation}
    \label{eq:lm_trace_ela_2}
    \sum_{j=1}^n \nu_j \p_{\bu}w_j = \frac{1}{2\mu+\lambda}\sum_{i=1}^n \widetilde{g}_i \nu_i.
  \end{equation}
  Inserting \eqref{eq:lm_trace_ela_2} into \eqref{eq:lm_trace_ela_1} gives 
  \begin{equation}
    \label{eq:lm_trace_ela_3}
    \p_{\bn}w_i|_{\p U} = \frac{1}{\mu}\left(\widetilde{g}_i - \frac{(\mu+\lambda)\nu_i}{2\mu+\lambda}\sum_{j=1}^n \widetilde{g}_j \nu_j \right).
  \end{equation}
  By virtue of the trace theorem, there exists $w_i\in H^{s+2}(U)$ such that $w|_{\p U} = 0$ and \eqref{eq:lm_trace_ela_3} is satisfied for $1\leq i\leq n$. And our claim is justified by $\bw = (w_1,\cdots,w_n)$. Finally, we can conclude that $\bu : = \bv+\bw \in \bH^{s+2}(U)$ and $\bu|_{\p U} = \bv|_{\p U} + \bw|_{\p U} = \bf$, $\bs(\bu)\bn|_{\p U} =  \bs(\bv)\bn|_{\p U} + \bs(\bw)\bn|_{\p U} =  \bg$.
\end{proof}

\begin{lemma}
  \label{lm:Green_ela}
  Let $s\geq 0$ and $\bu\in \tbH^{-s}(U)$, $\bv\in \bH^{s+2}(U)$, then we have
  \begin{align}
    \label{eq:Green_ela}
    &\inner{\div\bs(\bu), \bv}_{(\bH^{s+2}(U))^\ast,\bH^{s+2}(U)} - \inner{\bu,\div\bs(\bv)}_{(\bH^{s}(U))^\ast,\bH^{s}(U)}\\
    & = \inner{\bs(\bu) \bn, \bv}_{\bH^{-s-\frac{3}{2}}(\p U), \bH^{s+\frac{3}{2}}(\p U)} - \inner{\bu, \bs(\bv)\bn}_{\bH^{-s-\frac{1}{2}}(\p U), \bH^{s+\frac{1}{2}}(\p U)}.
  \end{align}
\end{lemma}
\begin{proof}
  Since $(C^\infty(\overline{U}))^n$ is dense in $\tbH^{-s}(U)$, we can find a sequence $\{\bphi_n\}_{n=1}^\infty$ in $(C^\infty(\overline{U}))^n$ such that $\bphi_n$ converges to $\bu$ in $\tbH^{-s}(U)$ norm. Thanks to \cite[Theorem 10.1.1]{Roitberg96}, we have $\div \bs(\bphi_n)$ converges to $\div\bs(\bu)$ in $\bH^{s+2}(U)^\ast$. Let $\{\bpsi_n\}_{n=1}^\infty$ be a sequence in $(C^\infty(\overline{U}))^n$ that converges to $\bv$ in $\bH^{s+2}(U)$ as $n\to \infty$. Then consequently, we have $\div\bs(\bpsi_n)$ converges to $\div\bs(\bv)$ in $\bH^{s}(U)$ and 
  \begin{align}
    \label{eq:Green_ela_1}
    &\quad \lim_{n\to \infty}\int_U \div\bs(\bphi_n)\cdot\bpsi_n  - \div\bs(\bpsi_n)\cdot \bphi_n\, dx \\
    & = \lim_{n\to \infty}\left(  \inner{\div\bs(\bphi_n), \bpsi_n}_{(\bH^{s+2}(U))^\ast,\bH^{s+2}(U)} -  \inner{\bphi_n,\div\bs(\bpsi_n)}_{(\bH^{s}(U))^\ast,\bH^{s}(U)}\right)\\
    & = \inner{\div\bs(\bu), \bv}_{(\bH^{s+2}(U))^\ast,\bH^{s+2}(U)} - \inner{\bu,\div\bs(\bv)}_{(\bH^{s}(U))^\ast,\bH^{s}(U)}.
  \end{align}
  Due to the trace theorem, we have $\bpsi|_{\p U}$ converges to $\bv|_{\p U}$ in $\bH^{s+\frac{3}{2}}(\p U)$ and $\bs(\bpsi_n)\cdot \bn|_{\p U}$ converges to $\bs(\bv)\cdot \bn|_{\p U}$ in $\bH^{s+\frac{1}{2}}(\p U)$, respectively. Hence, there holds
  \begin{align}
    \label{eq:Green_ela_2}
    &\inner{\bs(\bu) \bn, \bv}_{\bH^{-s-\frac{3}{2}}(\p U), \bH^{s+\frac{3}{2}}(\p U)} - \inner{\bu, \bs(\bv)\bn}_{\bH^{-s-\frac{1}{2}}(\p U), \bH^{s+\frac{1}{2}}(\p U)}\\
    &= \lim_{n\to \infty}\left( \inner{\bs(\bphi_n) \bn, \bpsi_n}_{\bH^{-s-\frac{3}{2}}(\p U), \bH^{s+\frac{3}{2}}(\p U)}-\inner{\bphi_n,\bs(\bpsi_n)\bn}_{\bH^{-s-\frac{1}{2}}(\p U), \bH^{s+\frac{1}{2}}(\p U)}\right)\\\
    &=\lim_{n\to \infty} \int_{\p U}\bs(\bphi_n)\bn\cdot \bpsi_n - \bs(\bpsi_n)\bn\cdot \bphi_n \ dx.
  \end{align}
  The Gauss-Green formula for elastic system (see e.g. \cite[eq.(4.22)]{sofonea2012}) reads
  \begin{align}
    \label{eq:Green_ela_3}
    &\int_{\p U}\bs(\bphi_n)\bn\cdot \bpsi_n - \bs(\bpsi_n)\bn\cdot \bphi_n dx - \int_U \div\bs(\bphi_n)\cdot\bpsi_n  - \div\bs(\bpsi_n)\cdot \bphi_n dx \\
    & = \int_ U (2\mu \be(\bphi_n) + \lambda \tr(\be(\bphi_n)\boldsymbol{I_n})) : \be(\bpsi_n) - (2\mu \be(\bpsi_n) + \lambda \tr(\be(\bpsi_n)\boldsymbol{I_n})) : \be(\bphi_n)\, dx\\
    & = 0.
  \end{align}
  Combining \eqref{eq:Green_ela_1}, \eqref{eq:Green_ela_2}, \eqref{eq:Green_ela_3} and taking $n\to \infty$, we can obtain \eqref{eq:Green_ela}.
\end{proof}
\newpage
\bibliography{reference}
\bibliographystyle{plain}

\end{document}